\documentclass[final]{siamltex}

\usepackage{latexsym, enumerate}
\usepackage{eepic}
\usepackage{epic}
\usepackage{graphicx}
\usepackage{color}
\usepackage{ifpdf}
\usepackage{amssymb,amsmath,epsfig, algorithm,algorithmicx,multirow}
\usepackage{amsfonts, dsfont}
\usepackage{subfigure}
\usepackage{bm}

\newtheorem{remark}{Remark}[section]

\title{Data-driven probability density forecast for stochastic dynamical systems
	\thanks{L.Jiang acknowledges the support of NSFC 12271408 and the
		Fundamental Research Funds for the Central Universities. }
}

\author{	Meng Zhao\thanks {School of Mathematical Sciences,  Tongji University, Shanghai 200092, China. ({\tt  1910736@tongji.edu.cn}).}
	\and	
	Lijian Jiang\thanks{School of Mathematical Sciences,  Tongji University, Shanghai 200092, China. ({\tt  ljjiang@tongji.edu.cn}).}
}

\begin{document}

	\maketitle
	
	\begin{abstract}In this paper, a data-driven nonparametric approach is presented for forecasting the probability density evolution of stochastic dynamical systems. The method is based on stochastic Koopman operator and extended dynamic mode decomposition (EDMD). To approximate the finite-dimensional eigendecomposition of the stochastic Koopman operator,  EDMD is applied to the training data set sampled from the stationary distribution of the underlying stochastic dynamical system.  The family of the Koopman operators  form a semigroup, which is  generated by the infinitesimal generator of the  stochastic dynamical system. A significant  connection between the generator and Fokker-Planck operator provides a way to construct an orthonormal basis of a weighted Hilbert space. A spectral decomposition of the probability density function is accomplished in this weighted space.  This approach is a data-driven method and  used  to predict  the probability density evolution and real-time  moment estimation.  In the limit of the large  number of  snapshots and  observables, the data-driven probability density approximation    converges to the Galerkin projection of the semigroup solution of Fokker-Planck equation on a basis adapted to an invariant measure.  The proposed method shares the similar idea  to diffusion forecast, but renders more accurate probability density  than the diffusion forecast does.  A few numerical examples are presented to  illustrate the  performance of the data-driven probability density forecast.
	\end{abstract}
	
	\begin{keywords}
		Stochastic dynamical system,  Stochastic Koopman operator, Probability density  decomposition, Diffusion forecast
	\end{keywords}
	
	
	\pagestyle{myheadings}
	
	\thispagestyle{plain}
	\markboth{M. Zhao,  L. Jiang}{Data-driven Probability Density Forecast}

	\section{Introduction}
	Dynamical systems afford a mathematical framework to characterize the evolution of a physical process and model  the abundant interactions between coupled physical quantities. Further, stochastic dynamical systems play a critical  role  in the analysis, prediction, and understanding of the behaviour of systems, or the flow maps that describe the evolution of stochastic proces. In this big data era, data are plentiful, while the governing equations or physical laws may be  elusive in the practical scenarios such as the  problems in climate science, finance and neuroscience. Inevitably, modern stochastic dynamical system is currently undergoing a renaissance, with analytical derivations  complemented by data-driven modeling \cite{data-driven}.  For the data-driven modeling in science and engineering, one of the fundamental problems is  to discover  high-dimensional dynamical systems from time-series data.

	Transfer operator  approach has  drawn significant   attention and been widely used for  data-driven modeling  in the past decades. In general, the transfer operator methods focus on the evolution of elements in observable spaces or measure spaces. These methods  may approximate  Koopman operator \cite{Koopman_RDS,Koopman1,Koopman2} and  Perron-Frobenius operator \cite{PF_Koopman,PF1}  by computing a finite-dimensional eigendecomposition of the operators. Remarkably, the action of Koopman operator  on the observables  lifts the action from the state space to the observable space  through composition of the observable and the underlying dynamical flow map. The Koopman operator transforms  the nonlinear, finite-dimensional dynamical system to a linear system in infinite dimensions.  Therefore, the modeling  of Koopman operator  follows the linear operator theory and the finite-dimensional approximation will be utilized in practical computation.

	Data-driven mechanism of approximating the transfer operators has  been widely investigated for a variety of applications, such as the identification of dynamical modes \cite{Koopman_RDS,dmdkf,PF_Koopman,DLdmd}, spectral analysis \cite{sa3,sa1,sa2}, nonparametric forecasting \cite{df1,df,sa4} and the modeling of slow-fast systems  \cite{sf1,sf2}.  The focus is on the estimation of the finite-dimensional projection of Koopman operator onto a specific subspace. This can be fulfilled by Dynamic Mode Decomposition(DMD) and its variants \cite{sDMD3,EDMD2,sDMD2,sDMD1,DMD1,kerDMD2,kerDMD1}. The DMD approach originates from  computational fluid dynamics \cite{DMD0}, and has been proven \cite{DMD01} that it gives the estimation of eigenvalues and  modes of the Koopman operator under certain conditions. In what follows, we will use one of many varieties - Extended Dynamic Mode Decomposition(EDMD) algorithm, which is robust and works   for stochastic dynamical systems \cite{EDMD1}, to approximate a finite-dimensional eigendecomposition of the Koopman operator. In EDMD, Koopman operator is projected to a finite-dimensional dictionary subspace of observables. Thus the eigenvalue problem becomes a matrix eigenvalue problem. In particular, the standard DMD (also known as exact DMD) is a specific case of EDMD for the dictionary space constituted by the state itself, which may not accurately characterize the action of Koopman operator. EDMD is able to give an accurate approximation to the Koopman operator  when the dictionaries are elaborately selected.
	The choice of the dictionaries is very technical and critical.  In the setting of ergodic dynamical systems, observable spaces are the nature $L^2$ spaces of the  scalar functions, and the measures are associated  with $L^2$ densities. Perron-Frobenius operator is the $L^2$-adjoint to Koopman operator. A common approach to approximate the Perron-Frobenius operator is celebrated as Ulam's method, and \cite{PF_Koopman} has shown that Ulam's method and EDMD are dual each other.

	For stochastic dynamical systems, the action of the Koopman operator  takes the expectation of the composition of observable and flow map, and  this Koopman operator is  called stochastic Koopman operator \cite{Koopman_RDS}. The  stochastic Koopman operator implies  that the action of stochastic Koopman operator on observables represents  the  solution of the Backward Kolmogorov Equation (BKE) of the stochastic dynamical system  with observables as initial condition. Furthermore, BKE is essentially governed by the infinitesimal generator of the stochastic dynamical system. Therefore, there is an elegant  connection bewteen stochastic Koopman operator  and the infinitesimal generator, i.e., the stochastic Koopman operator is an analytic and bounded semigroup generated by the infinitesimal generator of the stochastic dynamical system \cite{semigroup2}. Under appropriate conditions, the infinitesimal generator is $L^2$-adjoint to Fokker-Planck operator. Besides, the Fokker-Planck operator characterizes the transitions of probabilities, which reveals that the Fokker-Planck operator is the generator of the corresponding  Perron-Frobenius operator semigroup. It is worth noting  that, for an ergodic and time-reversible stochastic dynamical system with a unique invariant measure, the infinitesimal generator is self-adjoint and non-positive in the weighted  $L^2$  space associated with the invariant measure. The self-adjointness and non-positivity guarantee that the eigenvalues of the generator  are non-positive and real, and its spectrum are discrete  in this specific Hilbert space. Thus the  eigenfunctions of the generator constitute an orthonormal basis of weighted Hilbert space. Moreover, eigenvalues can be sorted by magnitude to extract the slow behaviours, which correspond  to larger eigenvalues.  The fast behaviours describe   noise. The above mentioned relationship motivates us to estimate the eigendecomposition of the Fokker-Planck operator through the data-driven method, and approximate the probability density functions in the space spanned by the eigenfunctions.
	
	In this paper, given a set of time-series data, our target is to approximate the probability density function of the stochastic dynamical systems. We adapt EDMD method to a weighted $L^2$ space to obtain the approximation of the Koopman operator. Hence  the eigendecomposition of the infinitesimal generator can be computed through the relationship bewteen the generator and semigroup. Then the adjoint property between the infinitesimal generator and Fokker-Planck operator provides a way to procure the eigenpairs of Fokker-Planck operator.  Finally, we decompose the probability density function in the space spanned by the eigenfunctions of the  Fokker-Planck operator, and use the Fokker-Planck equation to analytically compute the coefficient functions evolved in time. We call the proposed  method as dynamic probability density decomposition (DPDD). The approach allows that the time-series data  can be not only  from one single long-time trajectory, but also from ensemble trajectories. But, the utilization of EDMD requires that the data must be data pairs. As a similar method,  the diffusion forecast method proposed in \cite{df1,df} is also to estimate the probability density function using the eigenfunctions of the generator of a specific gradient flow system. It utilizes the diffusion maps  \cite{dm2,dm1} to approximate  the eigenfunctions. Comparing to diffusion forecast, our approach analytically compute the time-dependent coefficients of the decomposition without approximation error, which occurs in diffusion forecast. Our analysis and numerical results  show that DPDD achieves better accuracy and efficiency than the diffusion forecast.

	The paper  is organized as follows. In Section \ref{pre}, we briefly introduce the infinitesimal generator, Fokker-Planck operator,  Backward Kolmogorov equation and stochastic Koopman operator for stochastic dynamical systems. Dynamic probability density decomposition is elaborated in Section \ref{DPDD}. Besides, we also review the diffusion forecast method and make a comparison between diffusion forecast and the  Koopman operator-based DPDD. The convergence of DPDD is analyzed  in Section \ref{CA}.  A few numerical  results are presented to illustrate the probability density forecast using DPDD  in Section \ref{num}. Some conclusions and comments are made  in Section \ref{con}.
	
	\section{Preliminaries}\label{pre}
	In this section, we give a short review on stochastic dynamical systems and their  infinitesimal generator, and present the connection with  Koopman operator and its generator. These will be used in the whole paper.
	\subsection{The infinitesimal generator of SDE and its adjoint operator}\label{generator_FKO}
	Consider a time-homogeneous Markov process $X_t \in \mathcal{M}$ $\subseteq \mathbb{R}^d$,  which is generated by the SDE with drift term $b(x):\mathbb{R}^d \rightarrow \mathbb{R}^d$, diffusion function $\sigma(x):\mathbb{R}^d \rightarrow \mathbb{R}^{d\times s}$, and $s$-dimensional Wiener process $W_t$, i.e., the $X_t$ solves  the autonomous dynamical system,
	\begin{equation}\label{model}
		dX_t=b(X_t)dt+\sigma(X_t)dW_t,\quad X(0)=X_0.
	\end{equation}
	The (infinitesimal) generator $\mathcal {L}$ of SDE (\ref{model}) is then defined by
	\begin{equation}\label{gen}
		\mathcal{L}v=b\cdot \nabla v+\frac{1}{2}\Sigma:\nabla \nabla v,
	\end{equation}
	where $\Sigma(x)=\sigma(x)\sigma^T(x)$ is the diffusion  matrix. The formal $L^2$-adjoint operator $\mathcal{L}^*$ (a.k.a., the Fokker-Planck operator) is defined by
	\begin{equation}
		\mathcal{L}^*v=-\nabla \cdot \big(bv\big)+\frac{1}{2}\nabla \cdot\nabla\cdot \big(\Sigma v\big).
	\end{equation}
	\begin{theorem}\cite{ok,ms_t}(\textbf{Backward Kolmogorov Equation})
		Let $f :\mathcal{M} \rightarrow \mathbb{R}$ and define
		\[
		v(x,t)=\mathbb{E}\big(f(X_t)|X(0)=x\big),
		\]
		where the expectation is with respect to the Wiener process. If  $f$ is  smooth,  then the following equation
		\begin{equation}\label{bK}
			\left\{\begin{aligned}
				\frac{\partial v}{\partial t}&=\mathcal{L}v, \quad (x,t)\in \mathcal{M}\times (0,\infty),\\
				v(x,0) &= f(x),
			\end{aligned}\right.
		\end{equation}
		has a unique bounded classical solution $v(x,t)\in \mathcal{C}^{2,1}(\mathcal{M}\times \mathbb{R}^+)$.
	\end{theorem}
	
	We assume that the initial state $X_0$ is a random variable with probability density function $p_0(x)$, and the transition probability density function $p(x,t)$ exists and is a function in $\mathcal{C}^{2,1}(\mathcal{M}\times \mathbb{R}^+)$. Then $p(x,t)$ satisfies the \textbf{Fokker-Planck equation} \cite{GA1,ms_t,applied_sde}(a.k.a., the\textbf{ forward Kolmogorov equation}),
	\begin{equation}\label{FP}
		\left\{
		\begin{aligned}
			\frac{\partial p(x,t)}{\partial t} &=\mathcal{L}^* p(x,t), \quad (x,t)\in \mathcal{M}\times (0,\infty),\\
			p(x,0)&=p_0(x).
		\end{aligned}
		\right.
	\end{equation}
	Let the probability flux
	\begin{equation}\label{pro flux}
		J(x,t):=b(x)p(x,t)-\frac{1}{2}\nabla\cdot \big(\Sigma (x)p(x,t)\big).
	\end{equation}
	Then the Fokker-Planck equation can be written as
	\begin{equation*}
		\frac{\partial p(x,t)}{\partial t} +\nabla \cdot J=0.
	\end{equation*}
	Under periodic boundary condition and the assumption on the uniformly positive definition of diffusion matrix $\Sigma(x)$ \cite{GA1}, the stochastic process $X_t$ is ergodic. Then  the unique invariant distribution  $p_s(x)$ exists and solves
	\begin{equation}\label{sfp}
		\begin{aligned}
			\nabla \cdot J_s :=\nabla \cdot \big(b(x)p_s(x)\big)-\frac{1}{2}\nabla \cdot\nabla\cdot \big(\Sigma(x)p_s(x)\big)=0.
		\end{aligned}
	\end{equation}
	This implies the equilibrium probability flux is a divergence-free vector field. If we consider a reflecting boundary condition in a bounded domain for Fokker-Planck equation (\ref{FP}),
	then the stationary flux vanishes $\color{black}{J_s(x)=0}$.  This is a  stronger condition than (\ref{sfp}), since it involves $d$ equations, whereas the divergence-free equation (\ref{sfp}) is only one equation. Thus, in general, the  stationary probability flux does not vanish. If the stochastic process $X_t$ satisfies the detailed balance equation,
	\begin{equation}\label{DBE}
		p_s(x)p(y,t|x,s)=p_s(y)p(x,t|y,s),\quad \forall x,y\in \mathcal{M},\, s<t,
	\end{equation}
	then the stationary probability flux is zero. Therefore, for a time-reversible process,  the probability flux vanishes in steady state. In what follows, we will assume the time-homogeneous process $X_t$ is reversible. Conditions on the coefficients $b(x)$ and  $\Sigma(x)$ that are both necessary and sufficient for the detailed balance to be satisfied can be founded in \cite{handbook1,handbook2}, and this paper will do not  discuss it further.
	
	\begin{remark}
		The fact that the solution of Fokker-Planck equation is a probability density means that
		\[
		\int_{\mathcal{M}} p(x,t)dx=1, \quad p(x,t) \geq 0 \quad \forall t\in \mathbb{R}^+,
		\]
		and
		\[
		\int_{\mathcal{M}} p_s(x)dx=1,\quad p_s(x) \geq 0.
		\]
	\end{remark}
	
	Next, our objective is to establish the connection between  Koopman operator  and the infinitesimal generator. To accomplish this, we recall  Koopman operator in Subsection \ref{Koopman}.
	
	\subsection{Koopman operator}\label{Koopman}
	For a continuous-time deterministic system \cite{Koopman_RDS,PF_Koopman} defined on measure space ($\mathcal{M},\mathfrak{B},\nu$),
	\begin{equation*}
		\left\{\begin{aligned}
			\dot{x}&=\bm{F}(x), \quad x\in \mathcal{M},\\
			x(0)&=x_0,
		\end{aligned}
		\right.
	\end{equation*}
	where $\mathcal{M}\subseteq \mathbb{R}^d$ is the state space, $\mathfrak{B}$ is a sigma-algebra, and $\nu$ is a measure. The flow map (or evolution operator) is $\phi^{t}(x_0):=x(t,x_0)$.  Koopman operator $\mathcal{K}^t$ acts on the real-valued  observable of the state space, $\psi: \mathcal{M}\rightarrow\mathbb{R}$, and we define the Koopman operator  by
	\begin{equation}\label{K-def}
		\mathcal{K}^t\psi (x)=\psi \circ \phi^t(x)=\psi\big(\phi^t(x)\big),
	\end{equation}
	where $\circ$ denotes the function composition. If the observable $\varphi: \mathcal{M}\rightarrow \mathbb{R}$ and  satisfies
	\begin{equation}\label{K-eig}
		\mathcal{K}^t\varphi(x)=\mu(t)\varphi(x),
	\end{equation}
	then $\varphi(x)$ is an  eigenfunctions of the  Koopman operator and the corresponding $\mu(t)$ is the eigenvalue. The Koopman operator $\mathcal{K}^t$ is an infinite-dimensional linear operator and  describes the evolution of  observables, and lifts  the finite-dimensional nonlinear dynamical system $\phi^t$ to a linear system in infinite dimensions.
	
	For a continuous-time random dynamical system (RDS)   $(\mathcal{M},\mathcal{F},P,\Phi^t)$ \cite{Koopman_RDS}, the stochastic Koopman operator is defined as
	\begin{equation}\label{sto_K}
		\mathcal{K}^t\psi(x)=\mathbb{E}\big[\psi\circ\Phi^t(x)\big]=\mathbb{E}\big[\psi\big(\Phi^t(x)\big)\big],
	\end{equation}
	where $\mathbb{E}[\cdot]$ is the expectation value with respect to the probability measure $P$. For autonomous SDE (\ref{model}), the flow map $\Phi^t(x)$ in (\ref{sto_K})   is defined as
	\[
	\Phi^t(x):=x+\int_0^t b(X_s)ds+\sigma(X_s)dW_s:=X_t^x.
	\]
	The  Koopman operators mentioned previously are assume to be  time-homogeneous Markovian. For inhomogeneous time-varying random dynamical system $(\mathcal{M},\mathcal{F},P,\tilde{\Phi}^{t,t_0})$, the Koopman operator is defined by a two-parameter family of operator: $\mathcal{K}^{t,t_0}\psi(x)=\mathbb{E}\big[\psi\big(\tilde{\Phi}^{t,t_0}(x)\big)\big]$. We remark   the difference  between the definitions of two kinds of Koopman operators.  For time-homogeneous system, the Koopman operator is associated with the flow map $\Phi^{\Delta t}$ only dependent on the  time interval $\Delta t$, while the Koopman operator of time-varying system is determined by  a two-paramter flow map  $\tilde{\Phi}^{t,t_0}$ depending on both the current time $t$ and the initial time $t_0$. Since we focus on   the autonomous stochastic dynamical systems  in this work, the one-parameter Koopman operator family is only considered herein.
	
	Combining the backward Kolmogorov equation (\ref{bK}) and the definition (\ref{sto_K}) of Koopman operator, we have  the following relation between the Koopman operator and infinitesimal generator of SDE,
	\begin{equation}\label{Koopman&genrator}
		\mathcal{L}\psi=\mathcal{L}_K\psi:=\lim_{t\rightarrow 0^+}\frac{\mathcal{K}^t\psi-\psi}{t}, \quad \forall \psi\in \mathcal{F},
	\end{equation}
	where $\mathcal{L}_K$ is the generator that generates the Koopman semigroup. (\ref{Koopman&genrator}) implies that the infinitesimal generator of SDE is exactly the generator of Koopman operator. Thus  we will do not distinguish them, and denote them by   $\mathcal{L}$  from now on. We note  that the spetrum of generator $\mathcal{L}$ on a specific weighted Hilbert space defined in  Section \ref{DPDD}  is discrete, and lies along the negative real axis and tends to $-\infty$. The infinitesimal generator $\mathcal{L}$ is sectorial and the semigroup $\mathcal{K}$ generated by $\mathcal{L}$ is analytic and bounded. Given $e^{0\mathcal{L}}=I$, the following properties hold.
	\begin{proposition}\cite{semigroup2}\label{prop_koopman}
		\begin{itemize}
			\item[(i)] For $\forall s,t \geq 0$, $e^{s\mathcal{L}}e^{t\mathcal{L}}=e^{(s+t)\mathcal{L}}$;
			\item[(ii)] For $\forall t \geq 0$, and every $k \in \mathbb{N}$, $\mathcal{L}^k e^{t\mathcal{L}}=e^{t\mathcal{L}}\mathcal{L}^k$;
			\item[(iii)] For $\forall t > 0$ $\frac{d^k}{dt^k}e^{t\mathcal{L}}=\mathcal{L}^k e^{t\mathcal{L}}$.
		\end{itemize}
	\end{proposition}
	This indicates that there is a neat connection between the eigenpairs of Koopman operator $\mathcal{K}^t=e^{t \mathcal{L}}$ and generator $\mathcal{L}$,
	\begin{equation}\label{eig_relation}
		\mathcal{K}^t\varphi(x)=\mu(t)\varphi(x)=e^{t \mathcal{L}}\varphi(x)=e^{t \lambda}\varphi(x),
	\end{equation}
	where $(\lambda,\varphi)$ are the eigenpairs of  generator $\mathcal{L}$. Thus once we have the approximation of eigenpairs of Koopman operator, the eigendecomposition of generator $\mathcal{L}$ is available.
	
	\section{Dynamic Probability Density Decomposition}\label{DPDD}
	Assume that (\ref{DBE}) is satisfied as we discussed in Subsection \ref{generator_FKO}, then we have the vanished stationary probability flux and obtain the equality
	\begin{equation}\label{ps}
		b(x)p_s(x)=\frac{1}{2}\nabla\cdot \big(\Sigma(x)p_s(x)\big).
	\end{equation}
	Consider a set of functions $p(x,t) \in \mathfrak{D}(\mathcal{L}^*)$, and next we derive  a connection between the action of the infinitesimal generator on observables and the action of Fokker-Planck operator on distributions. This is a generalization of the relation for gradient flows presented in  \cite{book_diffusion1}.
	\begin{proposition}\label{adj_pro} Let $p(x,t)=f(x,t)p_s(x)$. Then the following holds,
		\begin{equation}\label{adj}
			\mathcal{L}^*p=p_s\mathcal{L}f.
		\end{equation}	
	\end{proposition}
	
	\begin{proof}  By the definition of Fokker-Planck operator, we have 	
		\begin{equation*}
			\begin{aligned}
				\mathcal{L}^*p&=\nabla \cdot \big(-b(x)p(x,t)\big)+\frac{1}{2}\nabla \cdot\nabla\cdot \big(\Sigma(x)p(x,t)\big)\\
				&=\nabla \cdot \big(-b(x)f(x,t)p_s(x)\big)+\frac{1}{2}\nabla \cdot\nabla\cdot \big(\Sigma(x)f(x,t)p_s(x)\big)\\
				&=-b(x)p_s(x)\nabla f-f \nabla\cdot\big(b(x)p_s(x)\big)\\
				&\quad+\frac{1}{2} \nabla \cdot\Big(\Sigma(x)p_s(x)\nabla f+f\nabla\cdot \big(\Sigma(x)p_s(x)\big)\Big)\\
				&=-b(x)p_s(x)\nabla f+\frac{1}{2}\Sigma(x)p_s(x):\nabla \nabla f+\nabla  f\nabla\cdot  \big(\Sigma(x)p_s(x)\big)\\
				&\quad -f\Big(\nabla \cdot \big[b(x)p_s(x)-\frac{1}{2}\nabla \cdot \big(\Sigma(x)p_s(x)\big)\big]\Big)\\
				&=\frac{1}{2}\Sigma(x)p_s(x):\nabla \nabla f+b(x)p_s(x)f\\
				&=p_s\mathcal{L}f,
			\end{aligned}
		\end{equation*}
		where the equality (\ref{ps}) is used  in the third  step.
	\end{proof}\\
	Besides, we also have $\mathcal{L}f=p_s^{-1}\mathcal{L}^*p=p_s^{-1}\mathcal{L}^*(fp_s)=p_s^{-1}\frac{\partial (fp_s)}{\partial t}=\frac{\partial f}{\partial t}$. This shows  that $f(x,t)$ satifies the backward Kolmogorov equation.
	
	In what follows we consider the weighted Hilbert $L^2$-space $L^2(\mathcal{M};p_s)$ weighted by the stationary distribution $p_s(x)$ with norm
	\begin{equation}\label{norm_weighetd}
		||f||^2_{0,p_s}=\langle f,f\rangle_{p_s}=\int_{\mathcal{M}} |f|^2p_s(x)dx.
	\end{equation}
	We claim that the generator $\mathcal{L}$ is self-adjoint and $-\mathcal{L}$ is positive and semi-definite    with respect to $L^2(\mathcal{M};p_s)$. The self-adjointness  can be verified by noting that, for any pair of functions $f,g \in L^2(\mathcal{M};p_s)$,
	\[
	\langle\mathcal{L}f,g\rangle_{p_s}=\langle\mathcal{L}f,gp_s\rangle=\langle f,\mathcal{L}^*(gp_s)\rangle=\langle f,p_s\mathcal{L}g\rangle=\langle f,\mathcal{L}g\rangle_{p_s},
	\]
	where we have  utilized  the fact (\ref{adj}) in the third equality. The positiveness  and semi-definite   can be obtained by noting that for any function $f \in L^2(\mathcal{M};p_s)$,
	\begin{equation*}
		\begin{aligned}
			\langle -\mathcal{L}f,f\rangle_{p_s}&=\langle -\mathcal{L}f,fp_s\rangle\\
			&=-\langle b(x)\cdot\nabla f+\frac{1}{2}\Sigma:\nabla\nabla f,fp_s\rangle\\
			&=-\frac{1}{2}\langle\nabla\cdot (\Sigma p_s)\nabla f,f\rangle-\langle\frac{1}{2}\Sigma:\nabla\nabla f,fp_s\rangle\\
			&=\frac{1}{2}\langle\Sigma\nabla f,\nabla f\rangle_{p_s}\\
			&=\frac{1}{2}||\sigma^T\nabla f||_{p_s}\geq 0.
		\end{aligned}
	\end{equation*}
	We assume that the generator $\mathcal{L}$ has a set of dicscrete spetrum. Then the eigenfunctions $\{\varphi_j\}_{j=1}^{\infty}$ form an orthonormal basis of $L^2(\mathcal{M};p_s)$. Therefore, for any observable $f\in L^2(\mathcal{M};p_s)$, the spectral expression $f(x,t)=\sum_{j}f_j(t)\varphi_j(x)$ with $f_j(t)=\langle f,\varphi_j\rangle$  holds. Recall the equality (\ref{adj}) $\mathcal{L}^*p=p_s\mathcal{L}f$, which immediately  indicates that  $ \mathcal{L}^*(\varphi_j p_s)=p_s\mathcal{L}\varphi_j=p_s \lambda_j\varphi_j=\lambda_j\varphi_j p_s$. Thus, $\{\varphi_j p_s\}_{j=1}^{\infty}$ are the eigenfunctions of the Fokker-Planck operator $\mathcal{L}^*$ with respect to eigenvalues $\{\lambda_j\}_{j=1}^{\infty}$. Moreover, the eigenfunctions of $\mathcal{L}^*$ form an orthonormal basis of space $L^2(\mathcal{M};p_s^{-1})$ following the orthonormality of $\{\varphi_j\}_{j=1}^{\infty}$,
	\[\langle\varphi_i p_s,\varphi_j p_s\rangle_{p_s^{-1}}=\langle\varphi_i ,\varphi_j \rangle_{p_s}=\delta_{ij}.
	\]
	
	\subsection{The spectral expansion}\label{sub_se}
	Given an orthonormal basis of $L^2(\mathcal{M};p_s^{-1})$, we can express the solution of the Fokker-Planck equation in (\ref{FP}) in the following form,
	\begin{equation}\label{se}
		p(x,t)=e^{t\mathcal{L}^*}p_0(x)=\sum_{i=0}^{\infty}c_i(t)\varphi_i(x)p_s(x).
	\end{equation}
	In order to obtain the system for the coefficients $c_i(t)$, we substitute (\ref{se}) into the Fokker-Planck equation (\ref{FP}) and utilize the Galerkin method, we have an ODE system for the  coefficients $c_i(t), \quad \forall i=0,1,\cdots,\infty$,
	\begin{equation}\label{coe}
		\begin{aligned}
			c_i(t)&=\langle e^{t\mathcal{L}^*}p_0(x),\varphi_i(x)p_s(x)\rangle_{p_s^{-1}}\\
			&=\langle p_0(x),e^{t\mathcal{L}}\varphi_i(x)\rangle\\
			&=\langle \sum_{j=0}^{\infty} c_j(0)\varphi_j (x) p_s(x) , e^{t\mathcal{L}}\varphi_i (x)\rangle\\
			&=\sum_{j=0}^{\infty}c_j(0)\langle \varphi_j (x),e^ {\lambda_i t} \varphi_i (x)\rangle_{p_s} \\
			&=\sum_{j=0}^{\infty} e^ {\lambda_i t} \delta_{ij} c_j(0)\\
			&=e^ {\lambda_i t} c_i(0).
		\end{aligned}
	\end{equation}
	Here the eigenpair $(\lambda_i, \varphi_i)$ are the eigenvalues and eigenfunctions of the generator $\mathcal{L}$, and $c_i(0)$ are the  coodinates of the initial condition $p_0(x)$ under the basis of eigenfunctions $\{\varphi_i p_s\}_{i=1}^{\infty}$. Thus combining (\ref{se}) and (\ref{coe})  we can  represent the probability density function $p(x,t)$ as follows,
	\begin{equation}\label{se1}
		p(x,t)=\sum_{i=0}^{\infty}e^ {\lambda_i t} \varphi_i(x)p_s(x)c_i(0),
	\end{equation}
	where $c_i(0)=\langle p_0(x),\varphi_i p_s\rangle_{p_s^{-1}}$.
	\par By the spectral expansion (\ref{se1}), we can obtain the probability density function without knowing the dynamical system (\ref{model}) and the manifold $\mathcal{M}$. In fact, all we need is that  the eigendecomposition of the generator $\mathcal{L}$, which can obtained by EDMD described in Subsection \ref{dis_pdf}. Next, the data-driven method (EDMD) is going to be introduced for  the numerical approximation of the spectral expansion.
	
	\subsection{Data-driven approximation}\label{EDMD}
	In this subsection, we briefly recall the extended dynamic mode decomposition (EDMD) for deterministic systems, and apply it to the stochastic dynamical systems  with ergodicity.
	\subsubsection{Extended Dynamic mode decomposition}\label{dEDMD}
	In this subsection, we elaborate EDMD \cite{EDMD2,EDMD1}, which  is a data-driven method to approximate the Koopman operator. Assume that we are given a data set of snapshots pairs $\{(x_m,y_m)\}_{m=1}^M$,
	\[
	\bm{X}=[x_1,x_2,\cdots,x_M],\qquad \bm{Y}=[y_1,y_2,\cdots,y_M]
	\]
	satisfying  $y_m=\Phi^{\Delta t}(x_m)$, where $\Delta t$ is the temporal interval of snapshots. Typically, the data samples  are not necessarily required to line on a single trajectory. EDMD requires that a dictionary of observables, $ \mathcal{D}=\{\psi_1,\psi_2,\cdots,\psi_N|\, \psi_i\in \mathcal{F} \}$, whose span is  denoted  as $\mathcal{F}_{\mathcal{D}}\subset\mathcal{F}$. For brevity, we define the vector-valued observable $\bm{\psi}: \mathcal{M}\rightarrow\mathbb{R}^{1\times N}$,  where $\bm{\psi}(x)=[\psi_1(x), \psi_2(x), \cdots, \psi_N(x)]^T$.
	EDMD constructs a finite-dimensional approximation $\bm{K}\in \mathbb{R}^{N\times N}$ of Koopman operator  by solving the least-squares problem,
	\begin{equation}\label{ls_edmd}
		\min_{\bm{K}\in \mathbb{R}^{N\times N}}\|\bm{K}\bm{\psi}(\bm{X})-\bm{\psi}(\bm{Y})\|_F^2=\min_{\bm{K}\in \mathbb{R}^{N\times N}}\sum_{m=1}^{M}\|\bm{K}\bm{\psi}(x_m)-\bm{\psi}(y_m)\|_2^2,
	\end{equation}
	where $\bm{\psi}(\bm{X})=[\bm{\psi}(x_1), \cdots, \bm{\psi}(x_M)], \quad \bm{\psi}(\bm{Y})=[\bm{\psi}(y_1), \cdots, \bm{\psi}(y_M)].$
	The $\bm{K}$  minimizes (\ref{ls_edmd}) is
	\[
	\bm{K}=\bm{\psi}(\bm{Y})\bm{\psi}(\bm{X})^{\dagger},
	\]
	where $\dagger$ denotes the pseudoinverse. This approach will require expensive computation for large $M$, since it needs the pesudoinverse of the $N\times M$ matrix $\bm{\psi}(\bm{X})$. Thus one usually takes another approach to compute $\bm{K}$,
	\begin{equation}\label{approx_K}
		\bm{K}= AG^{\dagger},
	\end{equation}
	where the relationship $\bm{\psi}(\bm{X})^{\dagger}=\bm{\psi}(\bm{X})^T\big(\bm{\psi}(\bm{X})\bm{\psi}(\bm{X})^T\big)^{-1}$ is  utilized  and
	\begin{equation}\label{AG}
		\begin{aligned}
			G&=\frac{1}{M}\sum_{m=1}^M \bm{\psi}(x_m)\bm{\psi}(x_m)^T,\\
			A&=\frac{1}{M}\sum_{m=1}^M \bm{\psi}(y_m)\bm{\psi}(x_m)^T.
		\end{aligned}
	\end{equation}
	As a result, we have $\bm{K}$ as a finite-dimensional approximation of  $\mathcal{K}$. Thus, if  $\xi_i\in \mathbb{R}^{N\times 1}$ is the $i$-th left eigenvector of $\bm{K}$ with the associated eigenvalue $\mu_i$, then the EDMD approximation of an eigenfunction of $\mathcal{K}$ is given by
	\[
	\phi_i(x)=\xi_i^T\bm{\psi}(x).
	\]

	\subsubsection{EDMD in weighted Hilbert space}\label{dis_pdf}
	Although EDMD was proposed for the deterministic systems and the deterministic Koopman operator in the beginning, it can also be applied to data from stochastic systems without changing the procedure. As shown in Subsection \ref{dEDMD}, EDMD can give  the eigenpair $(\mu_i,\varphi_i)$ of Koopman operator $
	\mathcal{K}$, whose eigenfunctions are  also the eigenfunctions of the generator $\mathcal{L}$ with respect to the eigenvalues $\lambda_i=\frac{\log \mu_i}{\Delta t}$ in Hilbert space $L^2(\mathcal{M})$. But we actually need the eigenpairs in the weighted Hilbert space $L^2(\mathcal{M};p_s)$, we modify EDMD to adapt to this  situation. The data set $\{x_i\}_{i=0}^M$ is sampling from stationary density function $p_s$, and EDMD will employ these data to approximate the eigendecomposition in weighted space. In this weighted Hilbert space, for $ \forall f,g \in L^2(\mathcal{M};p_s)$, we have an estimation for the inner product of two observables by  these stationary data,
	\[
	\langle f,g\rangle_{p_s}\approx \frac{1}{M}\sum_{i=0}^M f(x_i)g(x_i).
	\]
	To obtain the eigenpairs  of generator $\mathcal{L}$ , we utilize the data  observed at those sample points, i.e., the snapshot matrix is as follows,
	\begin{equation}\label{sp}
		\big(\bm{\psi}(x_0), \bm{\psi}(x_1),\cdots, \bm{\psi}(x_M)\big)=
		\begin{pmatrix}
			\psi_1(x_0)& \psi_1(x_1) & \cdots & \psi_1(x_M)\\
			\psi_2(x_0)& \psi_2(x_1) & \cdots & \psi_2(x_M)\\
			\vdots & \vdots & \ddots & \vdots\\
			\psi_N(x_0)& \psi_N(x_1) & \cdots & \psi_N(x_M)\\
		\end{pmatrix},
	\end{equation}
	where $\bm{\psi}=\big(\psi_1,\psi_2,\cdots,\psi_N\big)^T$ is the vector-valued observation function. Given
	\begin{equation}\label{sps}
		\begin{aligned}
			\bm{\psi}(X)&= \begin{pmatrix}
				\psi_1(x_0)& \psi_1(x_1) & \cdots & \psi_1(x_{M})\\
				\psi_2(x_0)& \psi_2(x_1) & \cdots & \psi_2(x_{M})\\
				\vdots & \vdots & \ddots & \vdots\\
				\psi_N(x_0)& \psi_N(x_1) & \cdots & \psi_N(x_{M})\\
			\end{pmatrix}, \\
			\bm{\psi}(Y)&=\begin{pmatrix}
				\psi_1(y_0)& \psi_1(y_1) & \cdots & \psi_1(y_{M})\\
				\psi_2(y_0)& \psi_2(y_1) & \cdots & \psi_2(y_{M})\\
				\vdots & \vdots & \ddots & \vdots\\
				\psi_N(y_0)& \psi_N(y_1) & \cdots & \psi_N(y_{M})\\
			\end{pmatrix},
		\end{aligned}
	\end{equation}
	where $y_i=\Phi^{\Delta t}(x_i), \, \Phi$ is the evolution operator associated with the stochastic dynamical system. We can get the approximation of $\Phi$ by using   Euler-Maruyama method, Milstein scheme, etc..  We apply the standard EDMD method in Subsection \ref{dEDMD} to above snapshot matrixs (\ref{sps}) and get  the eigendecomposition in weighted space.
	Ultimately, we obtain  the approximate dynamic probability density decomposition,
	\begin{equation}\label{pN}
		p_N(x,t)=\sum_{i=1}^{N}e^{\lambda_i t}\varphi_i(x)p_s(x)\tilde{c}_i(0),
	\end{equation}
	where $\color{black}{\tilde{c}_i(0)}$
	is an Monte-Carlo approximation of coefficients $c_i(0)=\langle p_0(x),\varphi_i p_s\rangle_{p_s^{-1}}=\mathbb{E}^0(\varphi_i)$ in the dynamic probability density decomposition (\ref{se1}). $(\lambda_i,\varphi_i)$ is the eigenpairs of generator $\mathcal{L}$, which is associated with the eigenpairs $(\mu_i,\varphi_i)$, acquired by EDMD, of Koopman operator.
	
	If we compute   $\color{black}{c_i(0)}$ by standard Monte-Carlo integration method, we may need the abundant independent samples drawn from the initial distribution $p_0(x)$. This  is computationally expensive and the initial distribution may be hard to sample. To address this intractable problem, we use the snapshots data $\{x_i\}_{i=0}^M$ sampled from stationary distribution to approximate the expectation by means of the importance sampling approach. The alternative distribution is the stationary distribution,  which is needed to be absolutely continuous with respect ro the initial distribution $p_0(x)$. Then the importance sampling estimator of $c_i(0)$ is given by
	\[
	\tilde{c}_i(0)=\frac{1}{M+1}\sum_{i=0}^M \varphi_i(x_i)p_0(x_i)/p_s(x_i).
	\]
	In this way, the variance of estimator will be zero when the stationary density is exactly equal to $p_0(x)\varphi_i(x)/c_i(0)$. Nevertheless, we can not really reach this zero-variance, since we do not know the accurate value of $c_i(0)$. If the stationary  distribution $p_s(x)$ is  proportional to the product of the initial density $p_0(x)$ and the basis function $\varphi_i(x)$, the variance can be significantly reduced.
	In the end, we summarize  our approach in  Algoritm \ref{DPDD_alg}.
	\begin{algorithm}[h]
		\caption{Dynamic Probability Density decomposition(DPDD)} \label{DPDD_alg}
		\raggedright{\bf Input:}  
		Data set of snapshots pairs $\{(x_i,y_i)\}_{i=0}^M$ sampling from invariant density $p_s(x)$, such that $y_i=\Phi^{\Delta t}(x_i)$; the initial condition $p_0(x)$. \\                           
		\raggedright{\bf Output:} The probability density function $p(x,t)$ and the eigenpairs $\{(\lambda_{i},\varphi_{i})\}$ of the infinitesimal generator $\mathcal{L}$.\\
		\raggedright \textbf{1}: Choose the appropriate observational space $\mathcal{F}_{\mathcal{D}}=span\{\psi_1,\psi_2,\cdots,\psi_N\}$, and generate the snapshots matrices $\bm{\psi}(X),\, \bm{\psi}(Y)$ given in (\ref{sps}).\\
		\raggedright \textbf{2}: Compute the Koopman matrix using formula (\ref{approx_K}) and (\ref{AG}), then the eigenpairs $\{(\mu_{i},\varphi_{i})\}$ can obtained by the eigendecomposition of Koopman matrix. And the eigenpairs $\{(\lambda_{i},\varphi_{i})\}$ of the infinitesimal generator $\mathcal{L}$ can be computed by the use of the relationship (\ref{eig_relation}).\\
		\raggedright \textbf{3}: Approximate the probability density function $p(x,t)$ by equation (\ref{se1}), with coefficients $\color{black}{\tilde{c}_i(0)=\frac{1}{M+1}\sum_{i=0}^M p_0(x_i)\varphi_i(x_i)/p_s(x_i)}$.
	\end{algorithm}
	\\

	\subsection{Diffusion forecast}\label{df}
	As another nonparametric forecasting method, diffusion forecast (DF)   \cite{df1,df,book_diffusion1} shares the similar idea to the proposed dynamic probability density decomposition, and treated  stochastic gradient systems with isotropic diffusion coefficient at the beginning. In the  diffusion forecast, diffusion map is a key technique  to acquire the approximate eigendecomposition of an operator,  which is the infinitesimal generator of a stochastic gradient flow with the invariant density function of system state as the potential energy. Then diffusion forecast uses the eigenfunctions of this specified generator as the basis to estimate the probability density function with a finite-dimensional truncation. In this subsection, we present the diffusion forecast and explore the relation between the diffusion forecast and the probability density forecast of Koopman operator.

	For the  diffusion forecast method, we consider  the same dynamical system as (\ref{model}).
	DF assumes that the system is ergodic with the unique equilibrium density $p_s(x)$  and given a time series data  $x_i:=x(t_i)\sim p_{s}(x)$. The goal of DF is to approximate $p(x,t)$ utilizing the given data without knowing the model. DF proposes using the diffusion maps algorithm to estimate the generator $\hat{\mathcal{L}}=\nabla \log \big(p_s\big)\cdot\nabla+\Delta $ of gradient flow,
	\begin{equation}\label{grad_flow}
		dx=\nabla \log\big(p_s(x)\big)dt+\sqrt{2}dW_t.
	\end{equation}
	The gradient flow system has the exactly same stationary distribution as the underlying system (\ref{model}). As a particular case of generator,  $\hat{\mathcal{L}}$ is naturally self-adjoint and non-positive with respect to the weighted space $L^2(\mathcal{M};p_s)$ as we have previously discussed in this section. By the  diffusion map algorithm, DF estimates the eigenpairs $(\hat{{\lambda}}_i,\hat{{\varphi}_i})$ of generator $\hat{\mathcal{L}}$, rather than those of generator $\mathcal{L}$ of underlying system (\ref{model}). Then the density function in DF is expressed as  the following form,
	\begin{equation}\label{se_df}
		p(x,t)=\sum_{k=0}^{\infty}\hat{c}_k(t)\hat{\varphi}_k(x)p_s(x),
	\end{equation}
	where the coefficients $\hat{c}_k(t)$ can be computed as follows,
	\begin{equation}\label{c_df}
		\begin{aligned}
			\hat{c}_k(t)&=\langle e^{t\mathcal{L}^*}p_0, \hat{\varphi}_kp_s\rangle_{{p_s}^{-1}}\\
			&=\langle p_0, e^{t\mathcal{L}}\hat{\varphi}_k\rangle\\
			&=\langle \sum_{j=0}^{\infty} \hat{c}_j(0)\hat{\varphi}_jp_s, e^{t\mathcal{L}}\hat{\varphi}_k\rangle\\
			&= \sum_{j=0}^{\infty} \langle \hat{\varphi}_j, e^{t\mathcal{L}}\hat{\varphi}_k\rangle_{p_s}  \hat{c}_j(0).
		\end{aligned}
	\end{equation}
	In practical computation, DF also truncates the summation  with a finite terms, and produces a matrix-vector multiplication, $\vec{\hat{c}}(t)=B\vec{\hat{c}}(0)$, with $k$-th component $\hat{c}_k(t)$. And the $kj$-th entry of the matrix $B$ is $B_{kj}=\langle \hat{\varphi}_j, e^{t\mathcal{L}}\hat{\varphi}_k\rangle_{p_s}$. Since the operator $\hat{\mathcal{L}}$ is the generator relevant to the  gradient flow rather than the underlying system (\ref{model}), DF  needs to approximate the coefficient matrix $B$. But our approach can compute the coefficients analytically and explictly by using the generator $\mathcal{L}$ of underlying system (\ref{model}).
	In DF method, the components of $B$ is numerically estimated as
	\begin{equation}\label{B}
		B_{kj}=\langle\hat{\varphi}_j, e^{t\mathcal{L}}\hat{\varphi}_k\rangle_{p_s}\approx \langle\hat{\varphi}_j, S_{t}\hat{\varphi}_k\rangle_{p_s}\approx\frac{1}{M-1}\sum_{i=1}^{M-1}\hat{\varphi}_j(x_i)\hat{\varphi}_k(x_{i+1}),
	\end{equation}
	where $S_{t}$ is a shift operator such that $S_{t}\hat{\varphi}_k(x_i)=\hat{\varphi}_k(x_{i+1})$.  We discover  that
	\[
	B_{kj}=\langle\hat{\varphi}_j, e^{t\mathcal{L}}\hat{\varphi}_k\rangle_{p_s}=\langle\hat{\varphi}_j, \mathcal{K}^t\hat{\varphi}_k\rangle_{p_s}=\langle\hat{\varphi}_j, \mathbb{E}\big(\hat{\varphi}_k\circ \Phi^t(x)\big)\rangle_{p_s}.
	\]
	Therefore  the shift operator $S_t$ actually follows the action criteria of deterministic Koopman operator, and is used as an approximation of the stochastic Koopman operator in diffusion forecast method.
	\begin{remark} \label{particular}
		If the underlying system (\ref{model}) is  a gradient flow with the isotropic diffusion,
		\[
		dx=-\nabla Vdx+\sqrt{2/\beta} dW_t,
		\]
		then DF has a trivial model with an analytic expression for the coefficient matrix $B$. In this  special case, the underlying system (\ref{model}) is exactly same as the gradient system (\ref{grad_flow}) with the generator $\hat{\mathcal{L}}$.
		We can derive that the stationary Gibbs distribution $p_s=\frac{1}{Z}\exp\big(-\beta V(x)\big)$ is a solution of time-homogeneous Fokker-Planck equation, where $Z=\int_{\mathcal{M}}e^{-\beta V(x)}dx$ is the normalization constant,
		\begin{equation*}
			\begin{aligned}
				p_s &=\frac{1}{Z}\exp\big(-\beta V(x)\big) \\
				\Rightarrow \quad \frac{\nabla p_s}{p_s}&=\frac{\nabla (e^{-\beta V})}{e^{-\beta V}}=\nabla \log (e^{-\beta V})=-\beta \nabla V\\
				\Rightarrow \quad \nabla p_s&=-\beta \nabla V\cdot p_s\\
				\Rightarrow \quad \mathcal{L}^*p_s&=\nabla\cdot(\beta \nabla V\cdot p_s)+\Delta p_s=0.
			\end{aligned}
		\end{equation*}
		For the  sake of simplicity, we assume that the normalization constant $Z=1$ and $\beta=1$. Thus
		\[
		\hat{\mathcal{L}}=\nabla \log(p_s)\cdot \nabla +\Delta=\nabla \log\big(\exp(-V)\big)\cdot \nabla +\Delta,
		\]
		and the generator of underlying system (\ref{model}),  $\mathcal{L}=-\nabla V\cdot \nabla+\Delta$.
		Consequently,  $ \mathcal{L}=\hat{\mathcal{L}}$, which can be estimated by the  diffusion map algorithm. Let $\mathcal{L}\hat{{\varphi}}_i=\hat{{\lambda}}_i\hat{{\varphi}}_i$.
		In this case, DF can achieve  the matrix $B$ analytically as the Koopman operator approach, i.e.,
		\[
		B_{kj}=\langle\hat{\varphi}_j, e^{t\mathcal{L}}\hat{\varphi}_k\rangle_{p_s}=\langle\hat{\varphi}_j, e^{t\hat{\mathcal{L}}}\hat{\varphi}_k\rangle_{p_s}=e^{\hat{\lambda}_k t}\delta_{jk}.
		\]
	\end{remark}
	
	In DF method, the role of diffusion map is to estimate the eigenpairs of generetor of gradient flow system (\ref{grad_flow}). In this scenario, the diffusion map  chooses an appropriate exponentially decaying function as the kernel $k_{\epsilon}$ with bandwidth $\epsilon$, employs the kernel density estimate method to approximate the stationary density, and reconstruct a new kernel by removing the sampling bias from the stationary distribution. Then we apply a weighted graph Laplacian normalization to construct a transition kernel (or transition density) $p_{\epsilon}$  corresponding  to a Markov chain. The  integral operator $P_{\epsilon}$ is defined by
	\[
	P_{\epsilon}f(x):=\int_{\mathcal{M}}p_{\epsilon}(x,y)f(y)p_s(y)dy.
	\]
	Then its generator $L_{\epsilon}f(x):=\frac{P_{\epsilon}f(x)-f(x)}{\epsilon}$ and  converges to the infinitesimal generator of SDE (\ref{grad_flow}), which is exactly what DF method seeks. Given a set  of data points $\{x_i\}_{i=1}^M$  sampled from the stationary distribution, the diffusion map algorithm furnishes the matrix approximation of the generator $\hat{\mathcal{L}}$, and the eigenpairs can  be computed by SVD.
	In what follows, we describe DF  and  the diffusion map algorithm in Algorithm \ref{DF_alg} and Algorithm \ref{DM_alg},  respectively.
	\begin{algorithm}[H]
		\caption{Diffusion forecast(DF)} \label{DF_alg}
		\raggedright{\bf Input:} Time series $\{x_i:=x(t_i)\}_{i=1}^M \sim p_s(x)$, the equilibrium measure $p_s(x)$ and initial density function $p_0(x)$.\\
		\raggedright{\bf Output:} The probability density function $p(x,t)$.\\
		\raggedright \textbf{1}: Compute the eigenpairs ${(\hat{\lambda}_i,\hat{\varphi}_i)}$ of a specific operator $\hat{\mathcal{L}}$, also the generator of gradient flow (\ref{grad_flow}) via diffusion maps algorithm outlined in Algorithm \ref{DM_alg}.\\
		\raggedright \textbf{2}: Estimate the probability density function $p(x,t)$ by formula (\ref{se_df}), (\ref{c_df}) and (\ref{B}) with $\hat{c}_j(0)=\frac{1}{M}\sum_{i=1}^M p_0(x_i)\hat{\varphi}_i(x_i)$.
	\end{algorithm}
	\begin{remark} Through the comparison between our approach (DPDD) and diffusion forecast, we summarize  the similarity and the difference between these two nonparametric methods. First of all, both of two methods require that the underlying system is ergodic and the given data sets are sampled from the unique invariant density function. One can ensemble the data from multiple trajectories, which can be realized by parallel simuluation (or observation), rather than a single long-time trajectory.  Secondly, for both approaches, the explicit expression of the invariant density is needed to fulfil an accurate estimation, while both can also work without knowing the invariant density through an appropriate method to approximate the density.  Thirdly, for our approach, the orthonormal basis are the eigenfunctions of generator $\mathcal{L}$  obtained from EDMD utilizing the relationship of Koopman operator $\mathcal{K}$ and SDE generator $\mathcal{L}$, while the basis set for diffusion forecast method is the eigenfunctions of generator of a gradient flow system, computed by the diffusion map algorithm. Therefore, there exists approximation error when computing coefficients $c_i(t)$ for the diffusion forecast, while DPDD method analytically solves the coefficients. Finally, diffusion forecast is designed for autonomous dynamical systems, while DPDD can work for non-autonomous stochastic  dynamical systems. For time-varying dynamical systems, the Koopman operator is dependent on time. Thus EDMD algorithm should be applied repeatedly for each temporal window where the dynamics  is assumed to be autonomous.
	\end{remark}

	\begin{algorithm}[h]
		\caption{Diffusion Map Algorithm \cite{dm2}} \label{DM_alg}
		\raggedright{\bf Input:} Data $\{x_i:=x(t_i)\}_{i=1}^M$ sampled from the invariant density $p_s(x)$ of gradient flow system. 
		\\
		\raggedright{\bf Output:} The approximate eigenpairs $\{(\hat{\lambda}_i,v_i)\}$ of generator of gradient flow.\\
		\raggedright \textbf{1}: Choose a kernel $k(x,y)=\exp (-\|x-y\|^2/{2\epsilon^2})$ with an appropriate smoothing paramter called the bandwidth $\epsilon$.\\
		\raggedright \textbf{2}: Compute $p_{\epsilon}(x)=\sum_{j}k(x,x_j)$, renormalize and create the matrix $\hat{K}$, such that, $\hat{K}_{ij}= \frac{k(x_i,x_j)}{\sqrt{p_{\epsilon(x_i)}p_{\epsilon(x_j)}}}$.\\
		\raggedright \textbf{3}:	Define $D_i=\sum_j\hat{K}_{ij}$ and construct a stochastic matrix $P=D^{-1}\hat{K}$, with entries $P_{ij}=\frac{\hat{K}_{ij}}{D_i}$.\\
		\raggedright \textbf{4}:  Define the generator matrix $\hat{L}=\frac{P-I}{\epsilon}$, and compute the first few eigenpairs of $\hat{L}$, s.t., $\hat{L}v=\hat{{\lambda}}v$.
	\end{algorithm}
	
	\begin{remark}
		We note that the integral operator $P_{\epsilon}$ defined in the diffusion map algorithm is exactly the stochastic Koopman operator. Assuming that the invariant measure $\nu$ is absolutely continuous with respect to the Lebesgue measure, and the Radon–Nikodym derivative is $p_s(x)$,  we have,
		\begin{equation*}
			\begin{aligned}	P_{\epsilon}f(x)&=\int_{\mathcal{M}}p_{\epsilon}(x,y)f(y)p_s(y)dy\\
				&=\int_{\mathcal{M}}p_{\epsilon}(x,y)f(y)d\nu\\
				&=\mathbb{E}\left[f\big(\Phi(x)\big)\right]\\
				&=\mathcal{K}f(x),
			\end{aligned}
		\end{equation*}
		where $\Phi(x)\sim p_{\epsilon}(x,\cdot)$ is the flow map of gradient flow system (\ref{grad_flow}). Therefore, the diffusion map algorithm constructs the stochastic Koopman operator via a very different way from EDMD.
	\end{remark}

	\section{Convergence Analysis}\label{CA}
	When we use  EDMD to approximate  the probability density function of  the   stochastic dynamical system, the total error stems from three sources. One source of error, addressed in section \ref{errA1}, arises from the finite-dimensional approximation in EDMD, which is actually  a Galerkin approximation of the Koopman operator \cite{PF_Koopman,EDMD1}. The convergence is achieved  by sufficiently large number of data.  The second source of error emerges when the spectral expansion is truncated at finite terms, and the number of truncated terms is equal to the dimensionality of observation space $\mathcal{F}_{\mathcal{D}}$, which is  spanned by the  independent dictionary functions. This kind of error actually originates from the projection of Koopman operator onto the finite-dimensional subspace of the observables \cite{EDMD2}. This error is  addressed in section \ref{errA2}. Since the stochastic Koopman operator is defined as the expectation value of composition of observables and evolution operator, the data-driven method should be implemented with the expectation of data pairs. However, the expectation  of observables are usually elusive. Thus, when EDMD is applied to the stochastic system and the given data is the snapshot pairs of realized trajectories instead of expectations pairs, another source of error emerges. This kind of error will be briefly discussed in section \ref{errA3}. Finally, we give the weak convergence of the DPDD approximate probability density function to the analytic solution of Fokker-Planck equation.

	\subsection{Convergence of EDMD to a Galerkin method}\label{errA1}
	Utilizing the Galerkin method, we obtain the matrices $\hat{A}$ and $\hat{G}$ whose entries
	\begin{equation}\label{Galerkin}
		\begin{aligned}
			\hat{A}_{ij}&=\int_{\mathcal{M}}\bm{\psi}_i\big(\Phi^{\Delta t}(x)\big)\bm{\psi}_j(x)d\nu(x)=<\mathcal{K}\bm{\psi}_i,\bm{\psi}_j>_{p_s},\\
			\hat{G}_{ij}&=\int_{\mathcal{M}}\bm{\psi}_i(x)\bm{\psi}_j(x)d\nu(x)=<\bm{\psi}_i,\bm{\psi}_j>_{p_s},\\
		\end{aligned}
	\end{equation}
	where $<p,q>_{p_s}=\int_{\mathcal{M}}p(x)q(x)p_s(x)dx$ is the inner product of the observable space. Then the finite-dimensional Galerkin approximation of Koopman operator would be $\hat{\bm{K}}=\hat{A}\hat{G}^{-1}$.
	
	Let the data points $\{x_i\}_{i=1}^M$ be  drawn from the invariant distribution $\nu$ on $\mathcal{M}$ with density $\rho$ and the number of snapshots $M \rightarrow \infty$.  We define the empirical measure $\hat{\nu}_M$ by
	\begin{equation}
		\hat{\nu}_M=\frac{1}{M}\sum_{m=1}^M\delta_{x_m},
	\end{equation}
	where $\delta_{x_i}$ is the Dirac function at $x_i$. In particular, the $ij$-th elements of $A$ and $G$ are as follows,
	\begin{equation}\label{EDMD-app}
		\begin{aligned}
			A_{ij}&=\frac{1}{M}\sum_{m=1}^M \bm{\psi}_i(y_m)\bm{\psi}_j(x_m)^T=\int_{\mathcal{M}}\bm{\psi}_i\big(\bm{F}(x)\big)\bm{\psi}_j(x)d\hat{\nu}_M,\\
			G_{ij}&=\frac{1}{M}\sum_{m=1}^M \bm{\psi}_i(x_m)\bm{\psi}_j(x_m)^T=\int_{\mathcal{M}}\bm{\psi}_i(x)\bm{\psi}_j(x)d\hat{\nu}_M.
		\end{aligned}
	\end{equation}
	The $ij$-th elements of $A$ and $G$ contain the sample means of $\bm{\psi}_i\big(\bm{F}(x)\big)\bm{\psi}_j(x)$ and $\bm{\psi}_i(x)\bm{\psi}_j(x)$. Thus, when $M$ is finite, (\ref{Galerkin})  is approximated by (\ref{EDMD-app}). By the law of large numbers, the sample mean almost surely converges to the expected value when the number of samples $M$ is sufficiently large. Further, we have the following convergence  with probability 1,
	\begin{equation}\label{cov}
		\lim_{M\rightarrow \infty}A_{ij} = \hat{A}_{ij},\quad \lim_{M\rightarrow \infty}G_{ij} = \hat{G}_{ij}.
	\end{equation}
	The convergence (\ref{cov}) could be read through two viewpoints. On the one hand, the convergence can be interpreted as the almost surely convergence of the empirical measure $\hat{\nu}_M$ to the distribution $\nu$ when there are sufficent data. On the other hand, we can consider the convergence as the result of the Monte-Carlo integration methods with independent identically distributed data. In this case, we have the convergence rate $\mathcal{O}(M^{-1/2})$.
	
	\subsection{Convergence of the finite-dimensional subspace projection}\label{errA2}
	The convergence of EDMD to the Galerkin method has  been discussed   in Section \ref{errA1}.
	In computation,  we use the finite-dimensional observable subspace projection $\mathcal{K}_N$ of Koopman operator. Now we consider  the convergence of the subspace operator $\mathcal{K}_N$ to Koopman operator $\mathcal{K}$. This convergence  was analyzed in  \cite{EDMD2}. Following the notation (\ref{norm_weighetd}),  we denote $L^2(\mathcal{M};p_s)$ norm of a function $f$ by $\|f\|_{0,p_s}$.
	Given certain assumptions, the strong convergence of $\mathcal{K}_NP_N^{\nu}$ to $\mathcal{K}$ is addressed as follows.
	\begin{theorem}\label{s_con}\cite{EDMD2} Assume that Koopman operator $\mathcal{K}: \mathcal{F}\rightarrow\mathcal{F}$ is bounded, and the observables $\psi_1,\psi_2,\cdots,\psi_N$ defining $\mathcal{F}_{\mathcal{D}}$ are selected from a given orthonormal basis of $\mathcal{F}$, i.e., $\{\psi_i\}_{i=1}^{\infty}$ is an orthonormal basis of $\mathcal{F}$.
		Then the sequence of operators $\mathcal{K}_NP_N^{\nu}$ converges strongly to $\mathcal{K}$ as $N\rightarrow \infty$, i.e.,
		\[
		\lim_{N\rightarrow \infty}\int_{\mathcal{M}}|\mathcal{K}_NP_N^{\nu}g-\mathcal{K}g|^2d{\nu}=0
		\]
		for all $g\in \mathcal{F}$.
	\end{theorem}

	We note that $P_N^{\nu}g$ is  the $L^2(\mathcal{M};p_s)$-projection of a function $g\in \mathcal{F}$ onto $\mathcal{F}_{\mathcal{D}}$. Besides, the weak convergence of the spectra of Koopman operator is given in the following theorem.
	\begin{theorem}\label{w_con}\cite{EDMD2}
		If $\mu_N$ is a sequence of eigenvalues of $\mathcal{K}_N$ with the associated normalized eigenfunctions $\varphi_N\in \mathcal{F}_{\mathcal{D}}$, i.e., $\|\varphi_N\|_{0,p_s}=1$, then there exists a subsequence $(\mu_{N_i},\varphi_{N_i})$ such that
		\[
		\lim_{i\rightarrow \infty} \mu_{N_i}=\mu,\quad \varphi_{N_i}\xrightarrow{w}\varphi,
		\]
		where $(\mu,\varphi)$ is the eigenpair of Koopman operator such that $\mathcal{K}\varphi=\mu\varphi$.
	\end{theorem}
	\subsection{Convergence of the  stochastic Koopman operator}\label{errA3}
	Now, we consider the stochastic system (\ref{model}) in Section \ref{pre}, and rewrite the model as
	\begin{equation}\label{model1}
		\frac{dX_t}{dt}=F\big(\eta_t,X_t),
	\end{equation}
	where $\eta_t$ is defined in form as the noise process $\eta_t:=\frac{dW_t}{dt}$.
	Given the  eigenpairs $(\lambda,\varphi)$ of the generator $\mathcal{L}$, we have the following equation,
	\begin{equation}\label{dvarphi1}
		\frac{d\varphi}{dt}=\lambda\varphi+\tilde{F}\cdot \nabla\varphi,
	\end{equation}
	where $\tilde{F}=F-\mathbb{E}(F)$. Similarly, when the model is written in the form of (\ref{model}), we obtain the following equation \cite{Koopman_RDS},
	\begin{equation}\label{dvarphi2}
		d\varphi=\lambda\varphi dt+\nabla\varphi \sigma dW_t.
	\end{equation}
	Thus, combining (\ref{dvarphi1}) and (\ref{dvarphi2}), we notice that $\tilde{F}=F-\mathbb{E}(F)=\sigma \frac{dW_t}{dt}=\sigma \eta$.
	From (\ref{Koopman&genrator}), for any observable $g$, we have $\mathcal{K}^tg(x)=\mathbb{E}[g\circ \Phi^t(x)]=e^{t\mathcal{L}}g(x)$ and $\mathcal{L}g(x)=\mathbb{E}(F)\cdot\nabla g$.  Thus a deviation arises when we use  path data pairs instead of  the expectation data pairs in EDMD  for the  approximation of stochastic Koopman operator,
	\begin{equation}\label{err3}
		\begin{aligned}
			g\circ\Phi^{\tau}-\mathcal{K}^{\tau}g&=g\circ \Phi^{\tau}-e^{\tau\mathcal{L}}g\\
			&=e^{\tau\mathcal{L}_d}g-e^{\tau\mathcal{L}}g\\
			&=(e^{\tau\mathcal{L}_d}-e^{\tau\mathcal{L}})g\\
			&\approx\mathcal{O}(\tau)\cdot(\mathcal{L}_d-\mathcal{L})g\\
			&=\mathcal{O}(\tau)\cdot\big(F-\mathbb{E}(F)\big)\cdot \nabla g\\
			&=\mathcal{O}(\tau)\cdot\sigma\eta\cdot\nabla g,
		\end{aligned}
	\end{equation}
	where $\mathcal{L}_d$ denotes the ``deterministic generator'' of (\ref{model1}), and satisfies $\mathcal{L}_dg(x)=F\cdot \nabla g$. For the approximate equality, we have used  the first-order differentiation approximation of semigroups $e^{\tau \mathcal{L}_d}$ and $e^{\tau \mathcal{L}}$ around $\tau=0$ in accordance with the term-by-term differential properties in Prop. \ref{prop_koopman}. The convergence of power series in $\color{black}{\cite{semigroup1}}$ ensures that the linear approximation of Koopman operator in (\ref{err3}) converges to zero. Here we omit the proof for brevity. From (\ref{err3}), we note  that the order of error $\mathcal{O}(\tau)$ is achieved for  the approximation of stochastic Koopman operator with  EDMD using  the snapshots pairs from stochastic systems.  The time $\tau$ is small enough so that the the first-order approximation of stochastic Koopman operator is accurate and tractable. This can be realized  in the numerical simulation when the snapshot pairs are the state evolution with small temporal intervals.
	
	\begin{theorem}
		If the number of snapshots $M$ and the dimension of finite-dimensional subspace projection $N$ go to infinity, i.e., $M,\,N\rightarrow\infty$, then  the following convergence holds,
		\[A_{ij}\xrightarrow{a.s.}\hat{A}_{ij},\quad G_{ij}\xrightarrow{a.s.}\hat{G}_{ij},\quad \mathcal{K}_{M,N}\xrightarrow{a.s.}\mathcal{K}_N\]
		and \[\mathcal{K}_N\xrightarrow{strongly}\mathcal{K} \quad \text{ in } L_2(\nu).
		\]
		Thus, the approximation $p_N(x,t)$ of the probability density function (\ref{se1}) weakly converges to $p(x,t)$ in $L^2(\mathcal{M};p_s^{-1})$, i.e.,
		\begin{equation}
			p_N(x,t)\xrightarrow{w} p(x,t), \quad\quad\forall (x,t)\in \mathcal{M}\times \mathbb{R}^+,
		\end{equation}
		where $p_N(x,t)=\sum_{i=1}^{N}\tilde{c}_i(0)e^{\lambda_{N_i} t}\varphi_{N_i}(x)p_s(x)$  and $(\lambda_{N_i},\varphi_{N_i})$ are the eigenpairs  in Theorem \ref{w_con}.
	\end{theorem}
	\begin{proof}Combining the expression (\ref{se1}) of dynamic probability density decomposition and the approximation $p_N(x,t)$, we have
		\begin{equation}
			\begin{aligned}
				p(x,t)-p_N(x,t)&=\sum_{i=1}^{\infty}c_i(0)e^{\lambda_{i} t}\varphi_{i}(x)p_s(x)-\sum_{i=1}^N \tilde{c}_i(0)e^{\lambda_{N_i} t}\varphi_{N_i}(x)p_s(x)\\
				&= \sum_{i=1}^N e^{\lambda_{N_i}t}\varphi_{N_i}(x)p_s(x)\big(c_i(0)-\tilde{c}_i(0)\big)\\
				&+\sum_{i=1}^N\big(e^{\lambda_it}\varphi_i(x)-e^{\lambda_{N_i}t}\varphi_{N_i}(x)\big)p_s(x)c_i(0)\\
				&+\sum_{i=N+1}^{\infty}e^{\lambda_it}\varphi_i(x)p_s(x)c_i(0)\\
				&\triangleq \uppercase\expandafter{\romannumeral1}+\uppercase\expandafter{\romannumeral2}+\uppercase\expandafter{\romannumeral3}.
			\end{aligned}
		\end{equation}
		Now, we analyze  the three terms  and show that all of them weakly converges to zero under the appropriate conditions. Firstly, $\tilde{c}_i(0)$ is a Monte-Carlo integration approximation of coefficients $c_i(0)$,  a strong convergence is ensured when $M\rightarrow \infty$. To address the second term $II$, we consider
		\[
		e^{\lambda_{i}t}\varphi_{i}(x)-e^{\lambda_{N_i}t}\varphi_{N_i}(x)= e^{\lambda_{i}t}\big(\varphi_{i}(x)-\varphi_{N_i}(x)\big)+\big(e^{\lambda_{i}t}-e^{\lambda_{N_i}t}\big)\varphi_{N_i}(x).
		\]
		By Theorem \ref{w_con}, we have the convergence of eigenvalues $\lambda_{N_i}\rightarrow\lambda_{i}$ and the weak convergence of eigenfunctions $\varphi_{N_i}\xrightarrow{w}\varphi_{i}$. Therefore, the weak convergence of second term  is also achieved.
		
		As for the third term $III$, we recall that the eigenvalues of the generator $\mathcal{L}$ is non-positive and descending sorted $\lambda_0=0>\lambda_1>\lambda_2>\cdots$ with exponential decay, and as the eigenfunctions of FK operator, $\varphi_{i}(x)p_s(x)$ can be considered as the smooth and bounded functions. In consequence, as dimensionality of observable space $N$ goes towards infinity, $e^{\lambda_{i}t}$ will decay  rapidly to zero, so we conclude that the third term  converges to zero. Therefore, the proof is completed.
	\end{proof}
	
	From the theorem above, we conclude that the spectral expansion of the probability density function will accurately approximate the truth density as the number $M$ of data snapshots and the number $N$ of observables  are larger enough.  The probability density can be used for the prediction of statistical moments, and provide the support to make decisions.

	\section{Numerical results}\label{num}
	In this section, some  numerical examples are given to illustrate the  efficacy of the DPDD method and make the comparison for the two nonparametric forecast methods: Koopman operator forecast (DPDD) and diffusion forecast (DF).    Subsection \ref{initial selection} is to show how the initial density function impacts on the approximation of DPDD approach. Subsection \ref{1dOU} demonstrates that if the underlying system is a gradient flow with isotropic diffusion, then the DPDD method  agrees with the DF method.  Both of them lead to  an analytic expression for the coefficients $c_i(t)$. Subsection \ref{2d_ex} and Subsection \ref{Lorenz63} show the comparison between DPDD and DF  using a  two-dimensional turbulence system and a noisy Lorenz-63 system,  respectively. In  Subsection \ref{realistic}, DPDD is used to forecast the ocean temperature, which is a  realistic problem.
	
	\subsection{The importance  of the initial density condition}\label{initial selection}
	As we have discussed in Section \ref{dis_pdf}, the importance sampling  is used  to avoid sampling from the initial distribution. While the importance sampling method is more accurate when the alternative distribution $p_s(x)$ is  proportional to  $p_0(x)\varphi_i(x)$. In this subsection, we will demonstrate that how the initial density $p_0(x)$ affects the approximation of  DPDD. For numerical simulation, we consider the following dynamical system,
	\begin{equation}\label{ex1}
		dX_t=-4X_t(X_t-1)(X_t+\frac{5}{4})dt+\sigma dW_t,
	\end{equation}
	where $\sigma$ is a constant.  The corresponding deterministic dynamical system ($\sigma=0$) has two stable equilibrium points $x=1,\, -\frac{5}{4}$ and one unstable point $x=0$. The stochastic system($\sigma=\sqrt{2}$)  also has two potential pits, although the trajectories deviate from the stable equilibrium points under the random perturbations. The invariant distribution $p_s(x)$ of the  stochastic system,
	\[
	p_s(x)=Z\exp\Big\{\frac{2}{\sigma^2}\int_{0}^{x}\big(-4u(u-1)(u+\frac{5}{4})\big)du\Big\}=Z\exp\Big\{-\frac{2}{\sigma^2}(x^4+\frac{1}{3}x^3-\frac{5}{2}x^2)\Big\},
	\]
	where $Z=\int_{\mathcal{M}}\exp\big\{\frac{2}{\sigma^2}\int_{0}^{x}\big(-4u(u-1)(u+\frac{5}{4})\big)du\big\}dx$ is the normalization constant. Figure \ref{p_s} shows the stationary density, from which we can see that there are two  equilibrium points with larger density and one unstable point with smaller density. The peaks of density function indicate that the stochastic state will trend to the stable points with large prabability, while the valley shows that the state will escape from  the unstable points.
	\begin{figure}[htbp]
		\centering
		\includegraphics[width=4.5in]{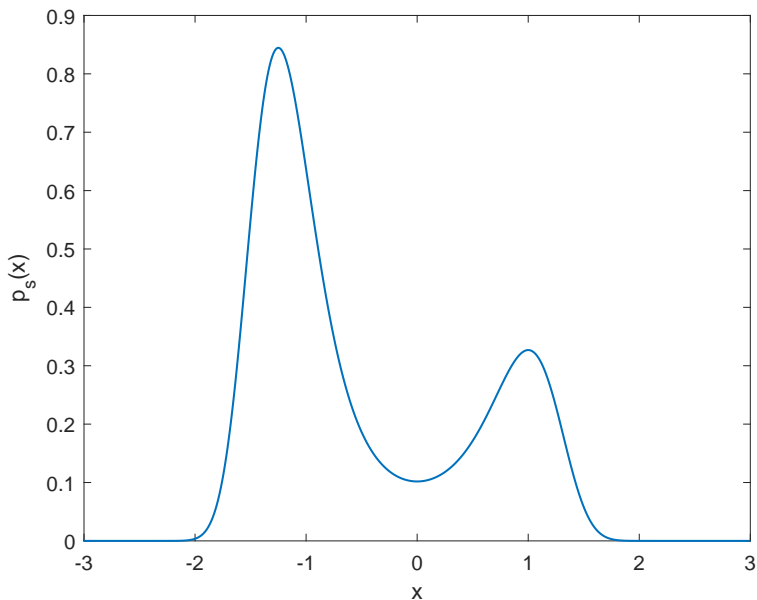}
		\caption{\textit{Stationary density function $p_s(x)$}}\label{p_s}
	\end{figure}
	
	\par We choose the  polynomial functions $\{x^n\}_{n=0}^{5}$  as the observation functions, sample $M=10000$ points from stationary density function $p_s(x)$, and utilize the snapshot data simulated by Euler-Maruyama scheme   to compute the eigenpair $(\lambda_i, \varphi_i (x))$ of Koopman operator $\mathcal{K}=e^{t\mathcal{L}}$, where $\mathcal{L}$ is the infinitesimal generator of the stochastic system (\ref{ex1}).  Figure \ref{sample} shows the histogram of the samples, and Figure \ref{samplepdf} depicts  the empirical probability density function,  which is acquired by a kernel density estimation method,  a built-in algorithm $\textit{ksdensity.m}$ in Matlab. The true stationary density is used as reference.
	\begin{figure}[htbp]
		\begin{minipage}[t]{0.5\linewidth}
			\centering
			\includegraphics[width=2.8in, height=2.3in]{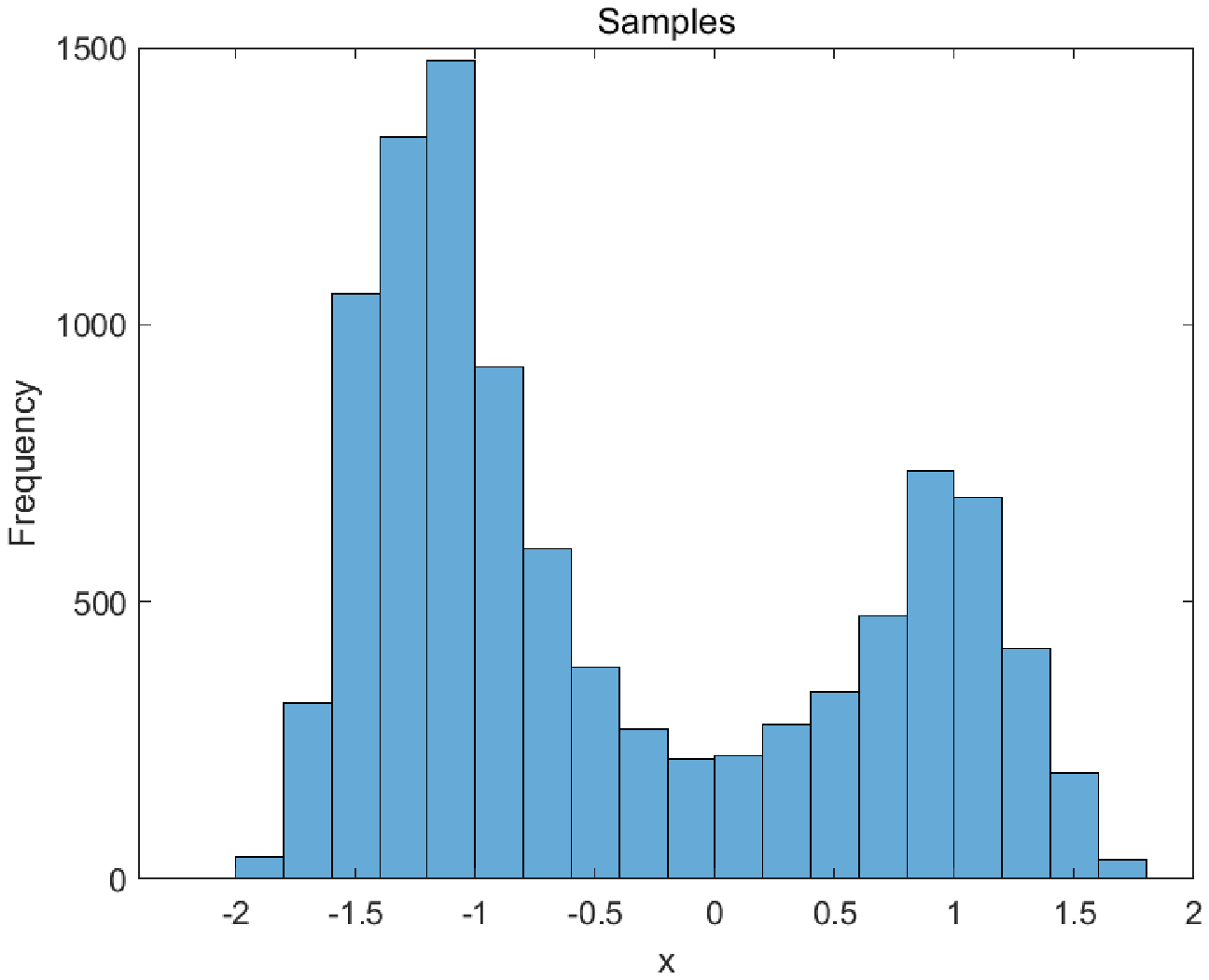}
			\caption{\textit{Samples from $p_s(x)$}}\label{sample}
		\end{minipage}
		\begin{minipage}[t]{0.5\linewidth}
			\centering
			\includegraphics[width=2.8in, height=2.3in]{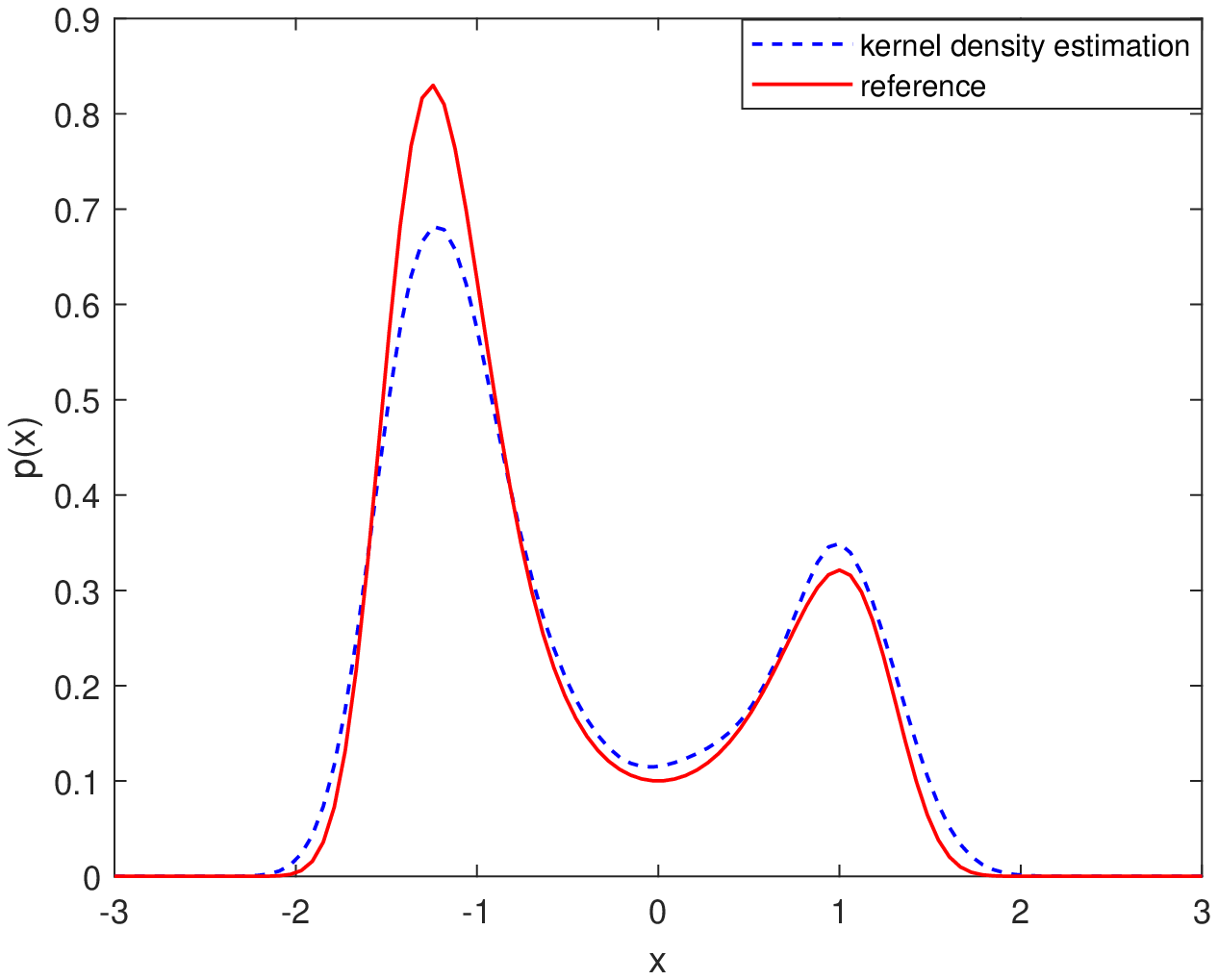}
			\caption{Empirical pdf of samples}\label{samplepdf}
		\end{minipage}
	\end{figure}
	
	Then we compute  the probability density function as (\ref{pN}). One can forecast the density at any time. In this example, we directly solve the Fokker-Planck equation of dynamical system (\ref{ex1}) as the reference solution. Firstly, let the initial density function $p_0(x)$ be exactly the stationary density. The approximate proability density function by DPDD and the reference solution are depicted  for six times  in Figure \ref{sol_close}, and the relative error is shown in Figure \ref{err_close}. From the above figures, we notice that the DPDD probability density  tends  to close to  the reference solution, and the relative error deceases as the time advances. The DPDD solution quickly converges to the truth solution because the initial density is chosen to be the  stationary density.
	\begin{figure}[htbp]
		\centering
		\includegraphics[width=5in]{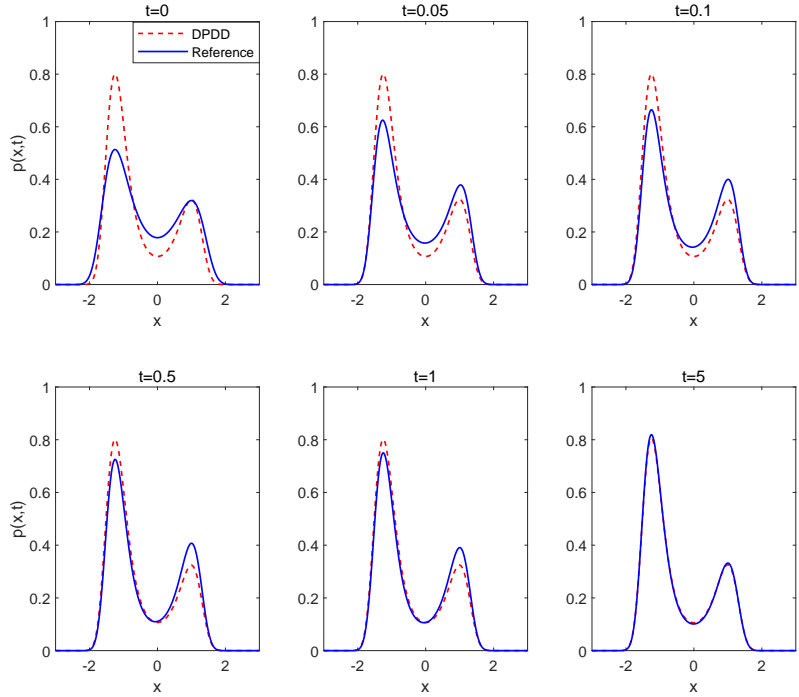}
		\caption{\textit{The approximate pdf at different times. }}\label{sol_close}
	\end{figure}
	
	\begin{figure}[htbp]
		\centering
		\includegraphics[width=5in]{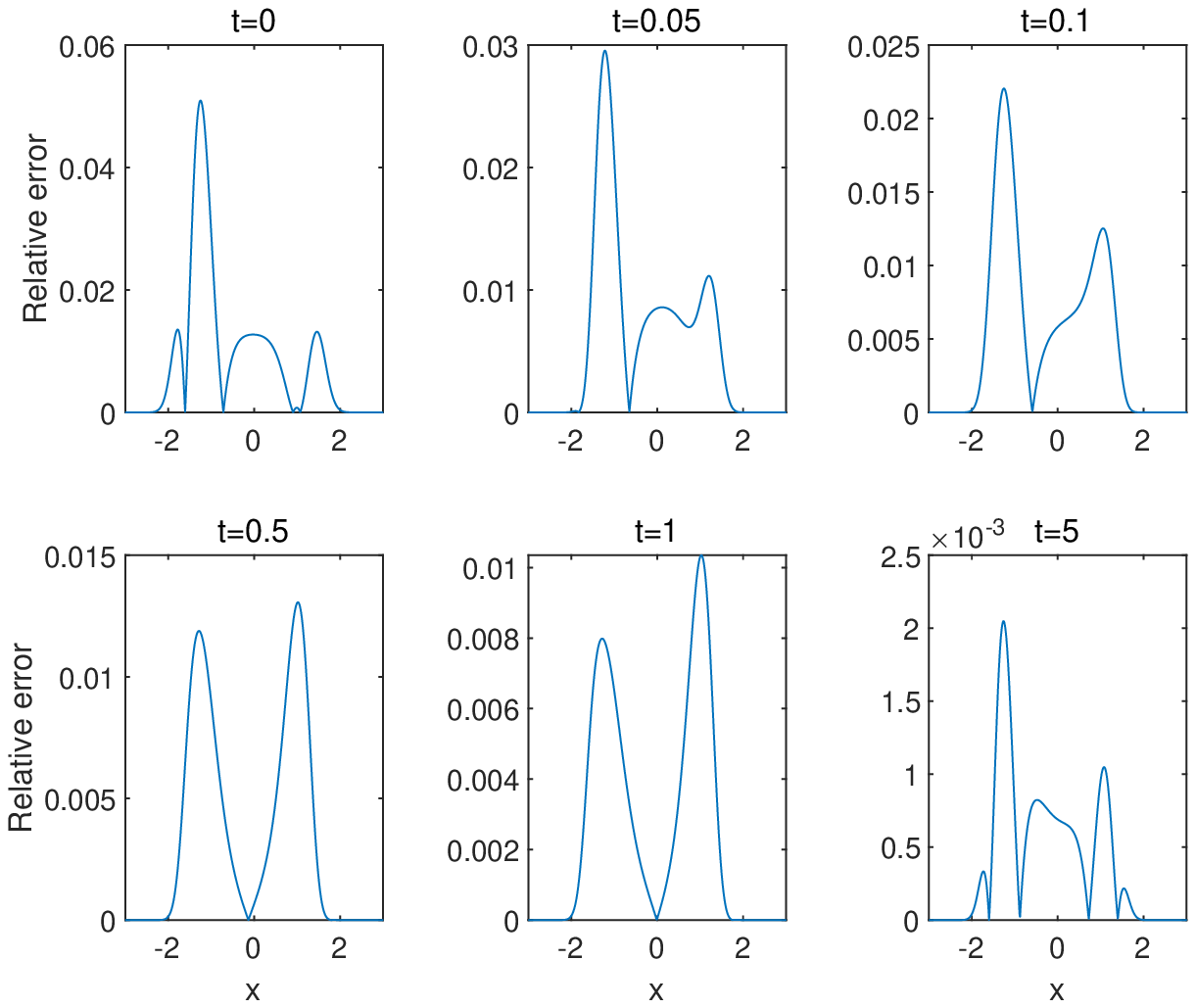}
		\caption{\textit{Relative error at different times. }}\label{err_close}
	\end{figure}
	
	\begin{figure}[htbp]
		\centering
		\includegraphics[width=5in]{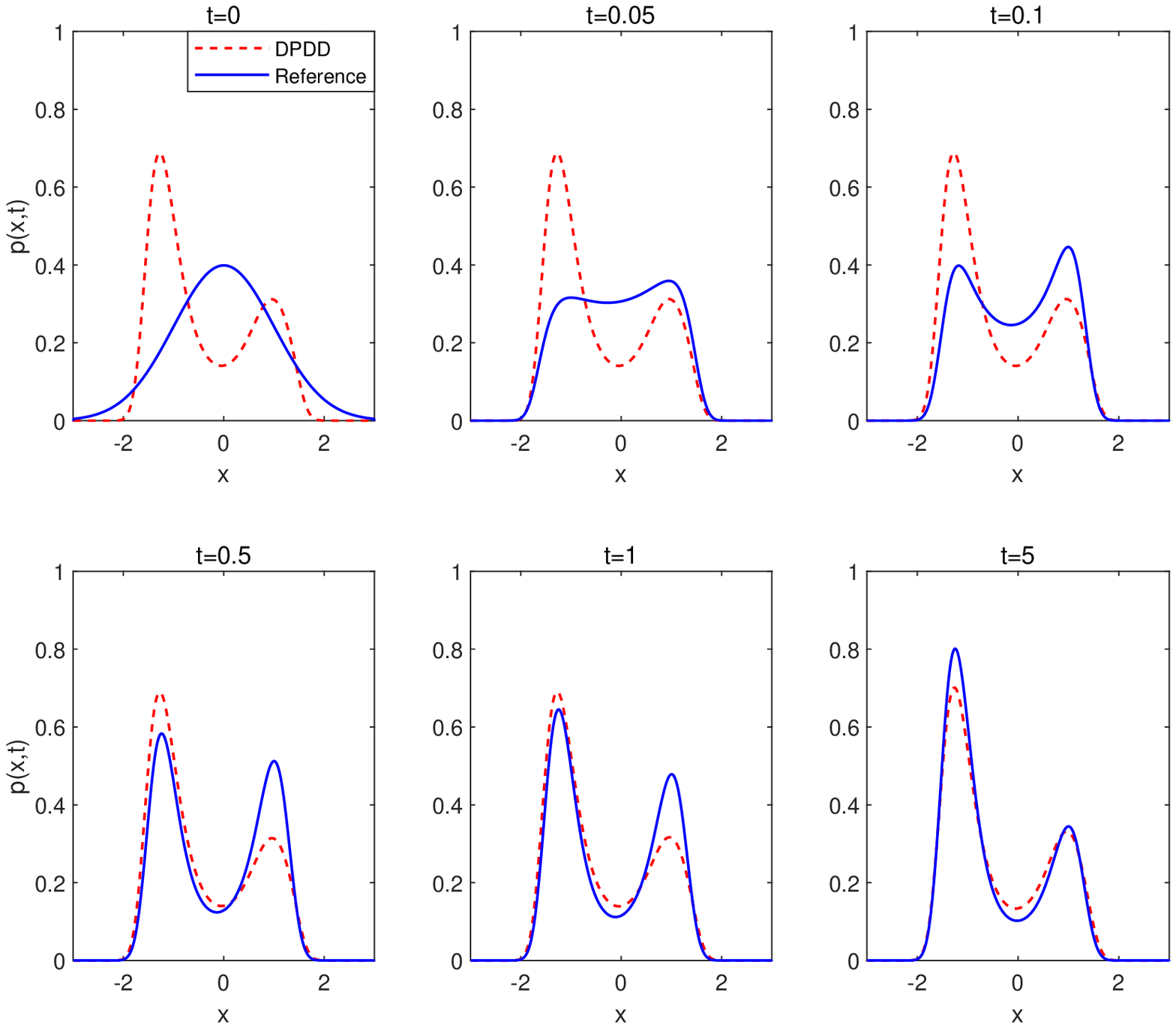}
		\caption{\textit{The approximate pdf at different times. }}\label{sol_diff}
	\end{figure}
	
	\begin{figure}[htbp]
		\centering
		\includegraphics[width=5in]{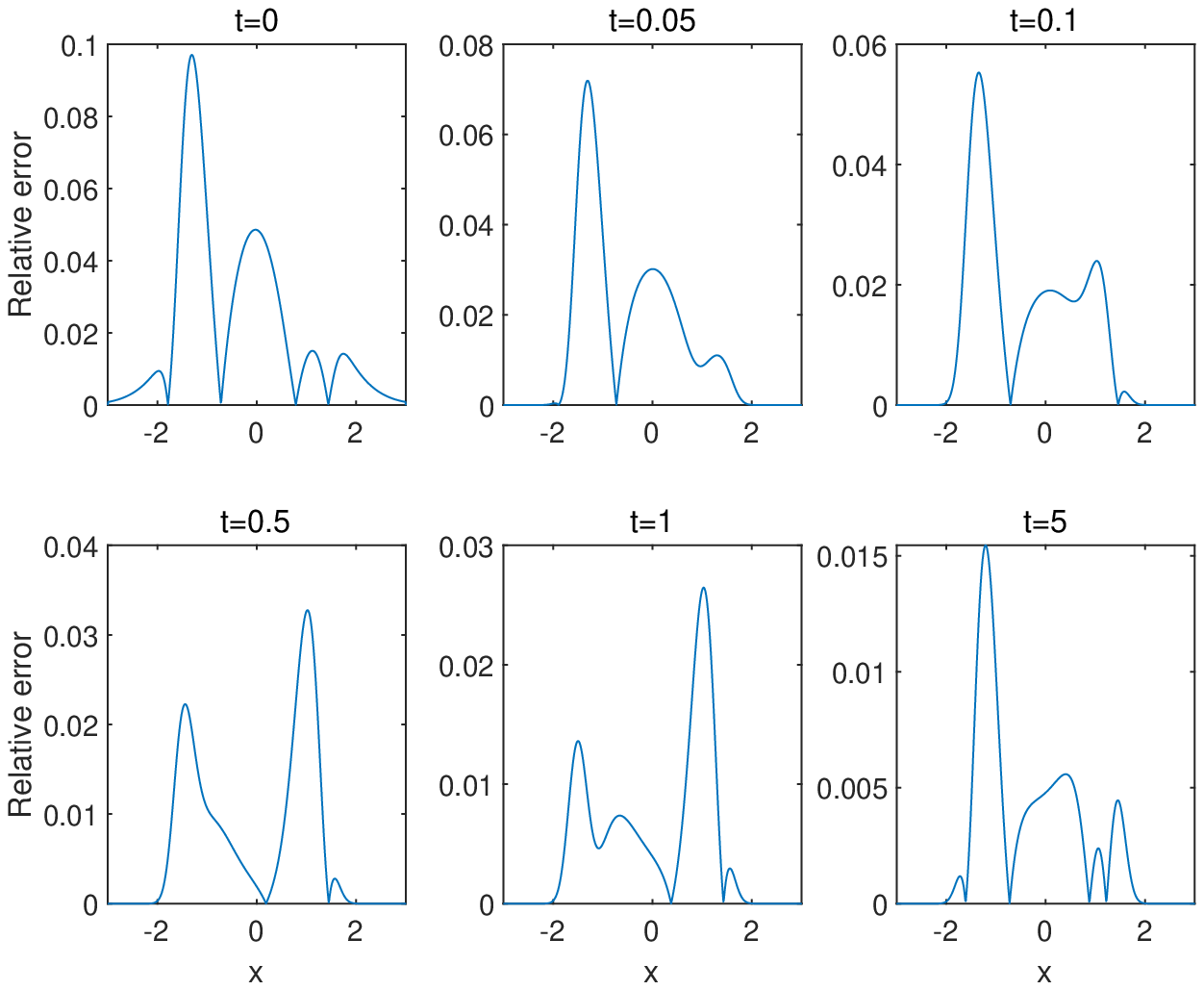}
		\caption{\textit{Relative error at different times. }}\label{err_diff}
	\end{figure}
	Another initial distribution $x_0 \sim \mathcal{N}(0,1)$ is selected for further comparison. The approximation solution and the relative error are showed in Figure \ref{sol_diff} and Figure \ref{err_diff},  respectively. In this case, the initial condition is very  different from the stationary density, so the approximation is less accurate than the previous case, even though the estimation still becomes accurate  as the time goes. Besides,  Figure \ref{err_close} and Figure \ref{err_diff}  show  that the approximation  error in the both cases  is larger at the peak and valley of the probability density function. This phenomena  often occurs  for multimodal modeling. The example confirms that if the initial density does not satisfy the proportional relation  in the importance sampling  approach, the  variance of integration estimation may be large and lead to a significant  DPDD approximation error.  Therefore,  the  initial density  impacts on the forecast accuracy.
	
	\subsection{Ornstein  Uhlenbeck process}\label{1dOU}
	As we clarified in the previous Remark \ref{particular}, the diffusion forecast (DF) method is consistent with our operator-theoretic approach DPDD when the underlying system is a gradient flow with isotropic diffusion. Therefore, in order to numerically  illustrate the identity,   we consider 1-dimensional OU process,
	\begin{equation}\label{OUmodel}
		dX_t=-\lambda X_tdt+\beta dW_t, \quad X_0=x_0.
	\end{equation}
	Take the parameter $\lambda=1$, $\beta=\sqrt{2}$, thus the analytic expression of
	probability density function is given by 
	\[
	p(x,t)=\frac{1}{\sqrt{2\pi v^2(t)}}\exp\Big(\frac{-\big(x-m(t)\big)^2}{2v^2(t)}\Big)
	\]
	with mean $m(t)=x_0 e^{-t}$ and variance $v^2(t)=1-e^{-2t}$, and the stationary states obey the standard normal distribution $
	\mathcal{N}(0,1)$, i.e., $p_s(x)=\frac{1}{\sqrt{2\pi}}e^{-\frac{x^2}{2}}$.
	\par  Moreover, the (infinitesimal) generator  $\mathcal{L}=-x\cdot\nabla+\Delta$. Then the eigenvalues of generator of  the OU process are the nonpositive integers $\lambda_{k}=-k, k=0,1,2,\cdots$, and  the corresponding eigenfunctions are the normalized  Hermite polynomials
	\begin{equation}\label{nor_Hermite}
		\varphi_{k}(x)=\frac{1}{\sqrt{k!}}H_k(x), \quad \text{ with } H_k(x)=(-1)^ke^{\frac{x^2}{2}}\frac{d^k}{dx^k}\big(e^{-\frac{x^2}{2}}\big),
	\end{equation}
	which form an orthonormal basis in $L^2(\mathbb{R};p_s)$. In particular, the first three eigenfunctions are $ \varphi_0=1,\, \varphi_1(x)=x,\,\varphi_2(x)=\frac{1}{\sqrt{2}}(x^2-1)$.
	\par In this example, we simulate the OU process along a long time by Euler-Maruyama scheme and take $10001$ states after a long-time enough evolution as the stationary samples. For the DPDD approach, we choose three monomial functions $\{\psi_n(x)=x^n\}_{n=0}^2$ as observable functions. Thus the data matrices in (\ref{sps}) are two $3\times 10000$ matrices, satisfying the equation $\Psi(Y)=\mathcal{K}\Psi(X)$. For the diffusion forecast method, $10000$ stationary samples are utilized, and three eigenfunctions are estimated by the diffusion map with variable bandwidth kernel.
	\begin{figure}[htbp]
		\centering
		\includegraphics[width=5.5in]{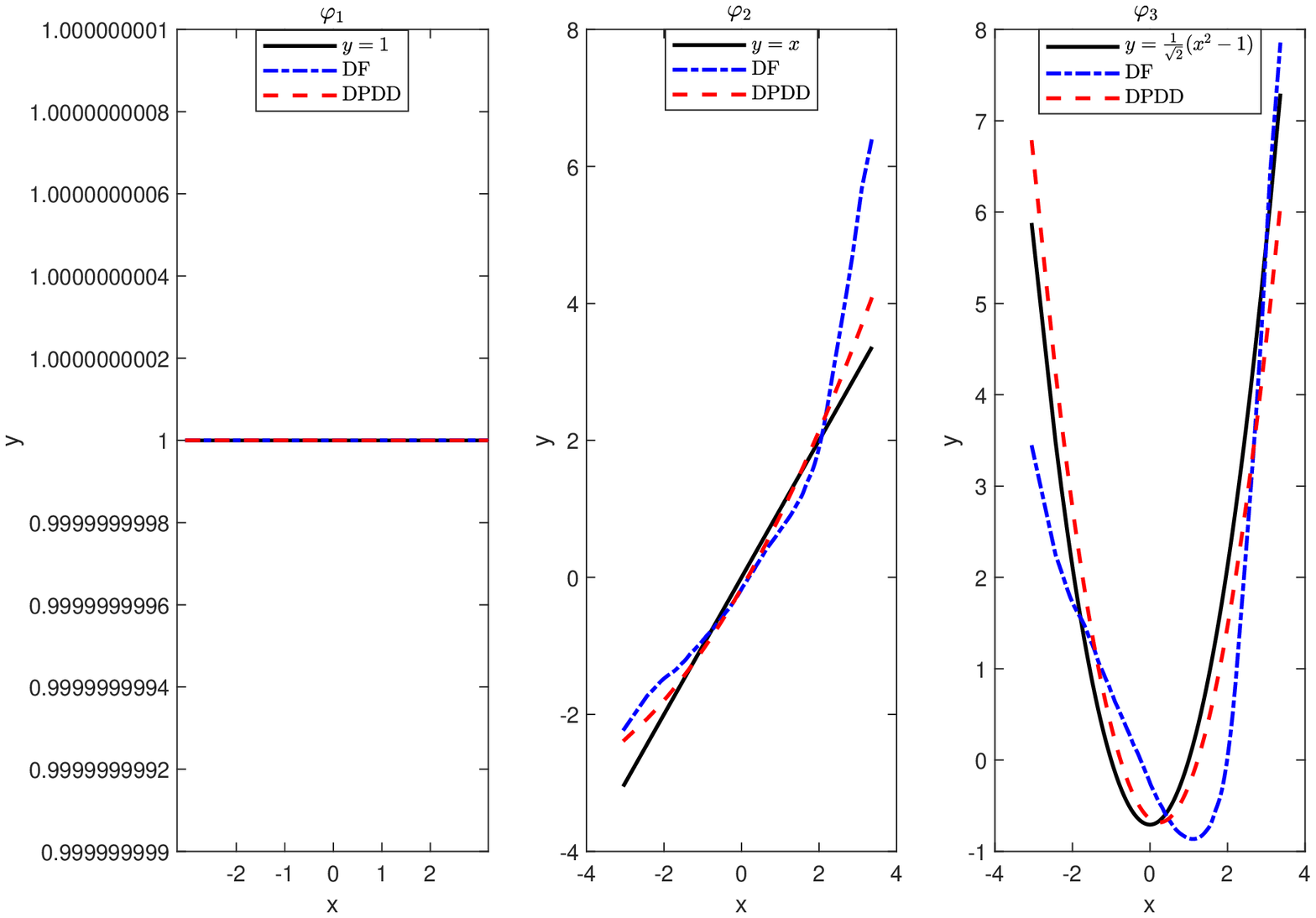}
		\caption{\textit{The first three eigenfunctions estimated by DPDD and DF methods, and the analytic eigenfunctions (the normalized Hermite polynomials) are given in black.}}\label{app_eig1dOU}
	\end{figure}
	\par 	In the following, we apply the two methods and show the eigenpairs, probability density functions and the first four moment functions to make a comparison. In Figure \ref{app_eig1dOU}, we plot the analytic first three eigenfunctions listed in (\ref{nor_Hermite}) as reference with the corresponding eigenvalues $\lambda_0=0,\,\lambda_1=-1,\,\lambda_2=-2$. The eigenfunctions approximated by two nonparametric methods are also illustrated in Figure \ref{app_eig1dOU}. We notice that the both methods can estimate the first constant eigenfunction exactly, however the DF method is less accurate for approximation of the second and third eigenfunctions than DPDD. Besides, DPDD also has a better approximation for the corresponding eigenvalues than DF as illustrated in  Table \ref{eigvalue_1dOU}.
	\begin{table}[h]
		\centering
		\caption{The first three eigenvalues estimated by DPDD and DF methods, with the analytic values as reference.}\label{eigvalue_1dOU}
		\begin{tabular}{|c|c|c|c|}
			\hline
			Eigenvalues & $\lambda_0$ & $\lambda_1 $ &  $\lambda_2$ \\
			\hline
			DF & $-1.1369\times 10^{-13}$ & -0.9163 & -1.4233\\
			\hline
			DPDD & $2.2204\times 10^{-14}$ & -1.0052 & -2.0974\\
			\hline
			Reference & 0 & -1 & -2 \\
			\hline
		\end{tabular}
	\end{table}
	\par To ensure the precision of  importance sampling Monte-Carlo integration method, the initial condition is taken to be a Gaussian distribution whose mean is randomly picked from the stationary distribution and variance is $0.5$. The behaviour of two approximate probability density functions at five different times is demonstrated in Figure \ref{pdf_1dOU}, and the analytic solutions are given in the third column as reference.
	\begin{figure}[htbp]
		\centering
		\includegraphics[width=5in]{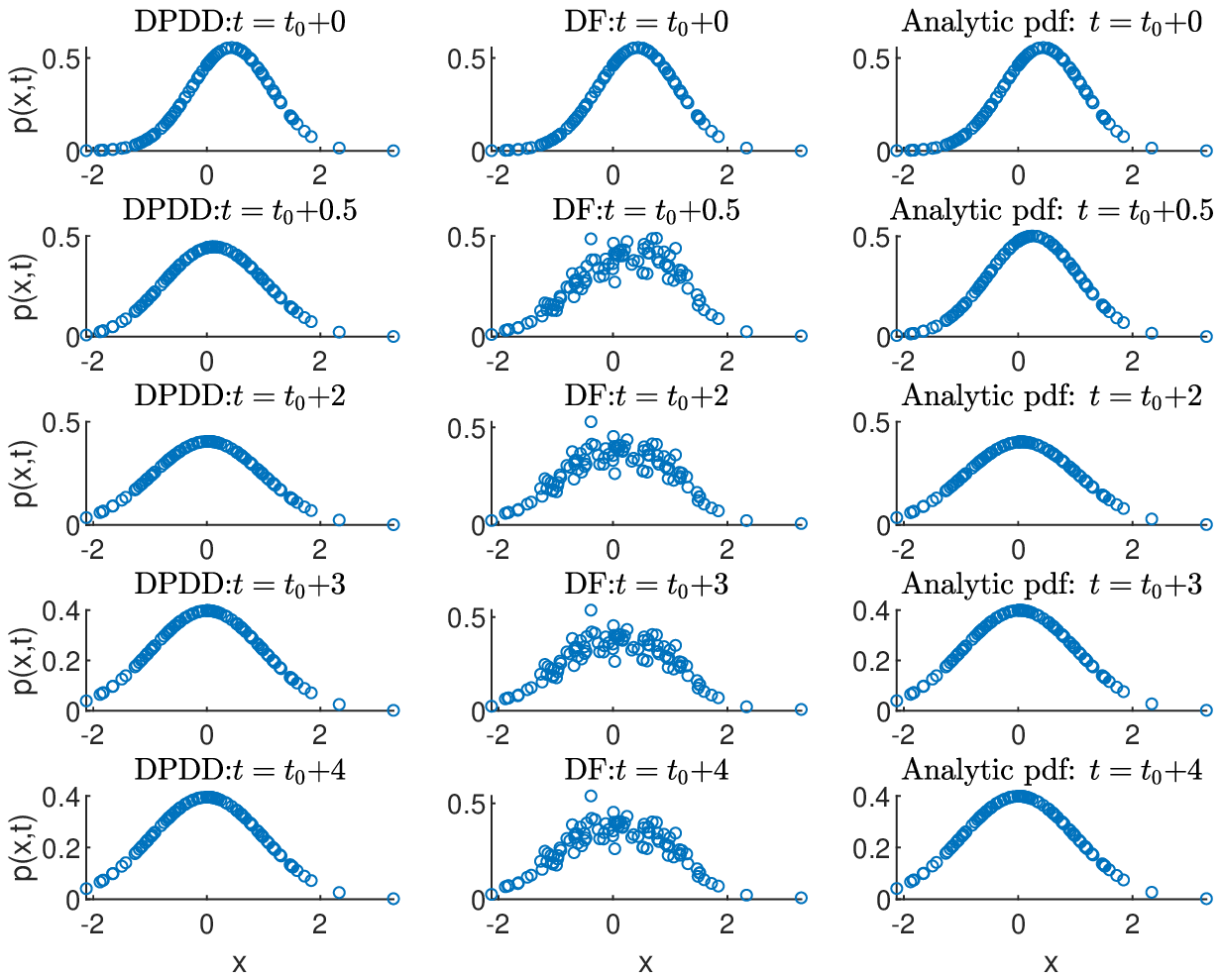}
		\caption{\textit{The approximate probability density function given by DPDD and DF methods at different times, and the third column is the analytic solution.}}\label{pdf_1dOU}
	\end{figure}
	Figure \ref{pdf_1dOU} shows  that, DPDD  almost obtains the exact solution, while the DF method leads  lots of deviations at many points. For the 1-dimensional OU process, the DPDD method gives a more accurate and robust approximation of the probability density function than the DF method.
	\begin{figure}[htbp]
		\centering
		\includegraphics[width=5in]{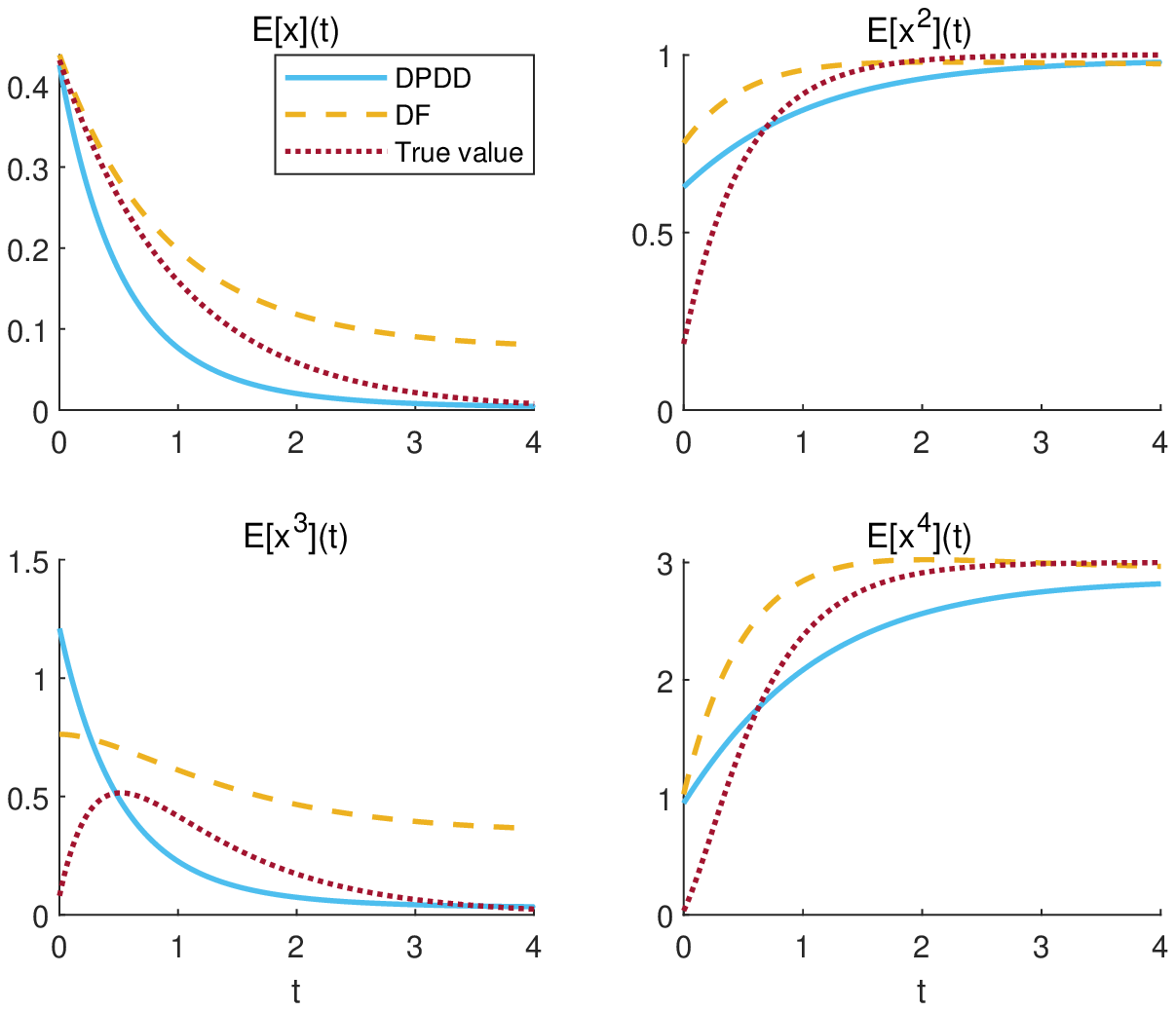}
		\caption{\textit{The evolution of the first four-order raw moments estimated by DPDD (blue line) and DF (yellow dashed line) methods, and the reference is ploted in red dotted line.}}\label{moments_1dOU}
	\end{figure}
	\par Finally, we plot the evolutions of the first four-order raw moments in Figure \ref{moments_1dOU} and the exact evolutions are directly computed by definition. From Figure \ref{moments_1dOU}, we observe that the both methods have a poor approximation at the beginning, since the initial condition is chosen very different from the true solution. Nevertheless, the approximations approach gradually to the truth evolution as the time goes. The moment functions given by DPDD are much closer to the analytic solutions and converges quickly as time passes, while the DF method even does not converge to the truth for the two odd order moments.  In summary,  DPDD can better approximate the probability density and the moments than DF does. Thus better prediction is achieved through DPDD. 
	
	In the previous two examples, the dynamical systems, with additive noise, are 1-dimensional gradient flows with some potentials. Then we can easily solve the analytic solution of stationary density by directly computing the time-homogeneous Smoluchowski equation \cite{GA1}. With the analytic density of stationary distribution, the spectral expansion (\ref{se1}) of the probability density in DPDD is more precise. However, the complex systems abound in the real world and there are often enormous challanges to simulate and characterize these systems. In what follows, we will consider the systems without knowing the truth  stationary density, and estimate the steady-state density by the variable bandwidth kernel method \cite{vbdm}. We will still compare our approach with diffusion forecast method in the following numerical results. Since the analytic probability density is difficult  to acquire, we will simulate plenty of trajectories of system states to get the sample data at any interested time, and use the corresponding point measure as reference, and this method is noted as an ensemble forecast \cite{book_diffusion1}. The region  where the sampling points are concentrated around corresponds to a large density, whereas the more sparse, the smaller the density.
	\subsection{A quadratic turbulence system}\label{2d_ex}
	In this subsection, we consider the following two-dimensional system of SDE  \cite{book_diffusion1},
	\begin{equation}
		\begin{aligned}
			\frac{du}{dt}&=\frac{1}{2}uv-d\Lambda_{11}u+(1-d\Lambda_{12})v+S_{11}\dot{W_1}+S_{12}\dot{W_2}\\
			\frac{dv}{dt}&=-\frac{1}{2}u^2+(-1-d\Lambda_{12})u-d\Lambda_{22}v+S_{12}\dot{W_1}+S_{22}\dot{W_2},
		\end{aligned}
	\end{equation}
	where the nonlinear terms conserve energy and $\dot{W_i}$ denotes the independent white noise. This model is a special case of the paradigm model with persistently unstable dynamics for turbulence introduced in \cite{turbulence}. Here, we set the parameters in the model $d=1/2,\, S=\Lambda^{1/2}$ with
	\[
	\Lambda=\begin{pmatrix}
		1 & 1/4\\
		1/4 & 1\\
	\end{pmatrix}.
	\]
	Euler-Maruyama scheme is utilized to simulate the solutions of  this model with time discretization $\tau=0.01$. Using $M=20000$ snapshots data $\{x_i=(u_i,v_i)\}_{i=1}^M$ and $4$-dimensional  observable $\bm{g}=\big(1,u,v,(u+v)^2\big)^T$, we approximate the stochastic Koopman operator, the Koopman eigenfunctions and eigenvalues by applying EDMD in the weighted space, and the  eigenfunctions are evaluated at sample points. Once we obtain the eigenfunctions, we use them as the basis functions to get the spectral expansion of the probability density funtion as shown in Section \ref{DPDD}. The number of truncation terms equals to the number of  first few eigenfunctions, which is exactly  the dimensions of observable $\bm{g}$. The probability density function can be used to forecast and make a strategic decision.
	\par As stated before, we compare our approach with the DF method, whose $1000$ basis functions are constructed by diffusion maps algorithm, and an ensemble forecast is adopted as reference. In Figure \ref{pdf_2d}, we plot the probability density function at five times  $t=0,\, 0.5,\,2,\,3,\,4$ obtained from different methods.
	\begin{figure}[htbp]
		\centering
		\includegraphics[width=5in]{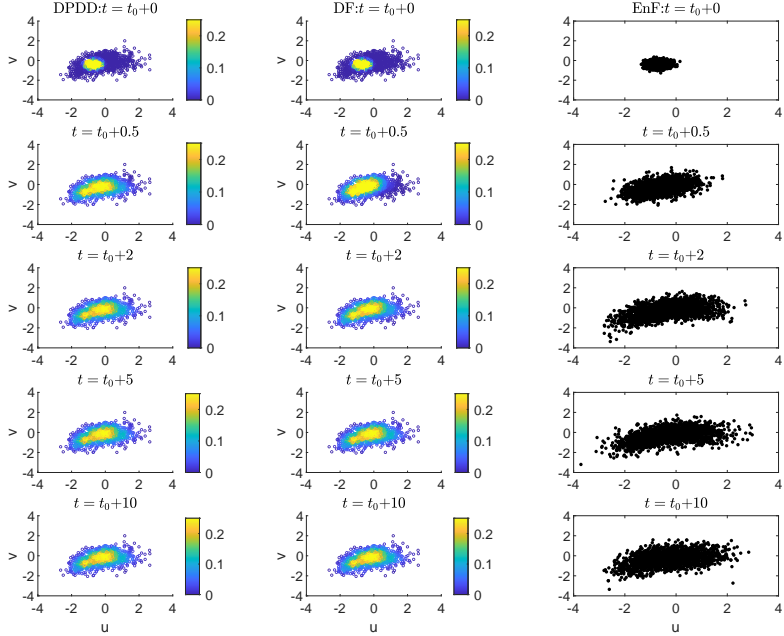}
		\caption{\textit{The approximate probability density function at different times. The left column is estimated by DPDD, the middle is estimated by DF, and the right column (ensemble forecast) is given as reference.}}\label{pdf_2d}
	\end{figure}
	In the first two colums of Figure \ref{pdf_2d}, the color represents the value of density, and the yellow corresponds to larger value, while the blue is associated with smaller value. The region  with dense points has large values of probability density in the last column. As we can see, the probability density fucntions approximated by two methods are very close to each other, and both match the distribution computed  by ensemble forecast. The evolution of the first four raw moments of dynamical state are depicted as following.
	\begin{figure}[htbp]
		\centering
		\includegraphics[width=5in]{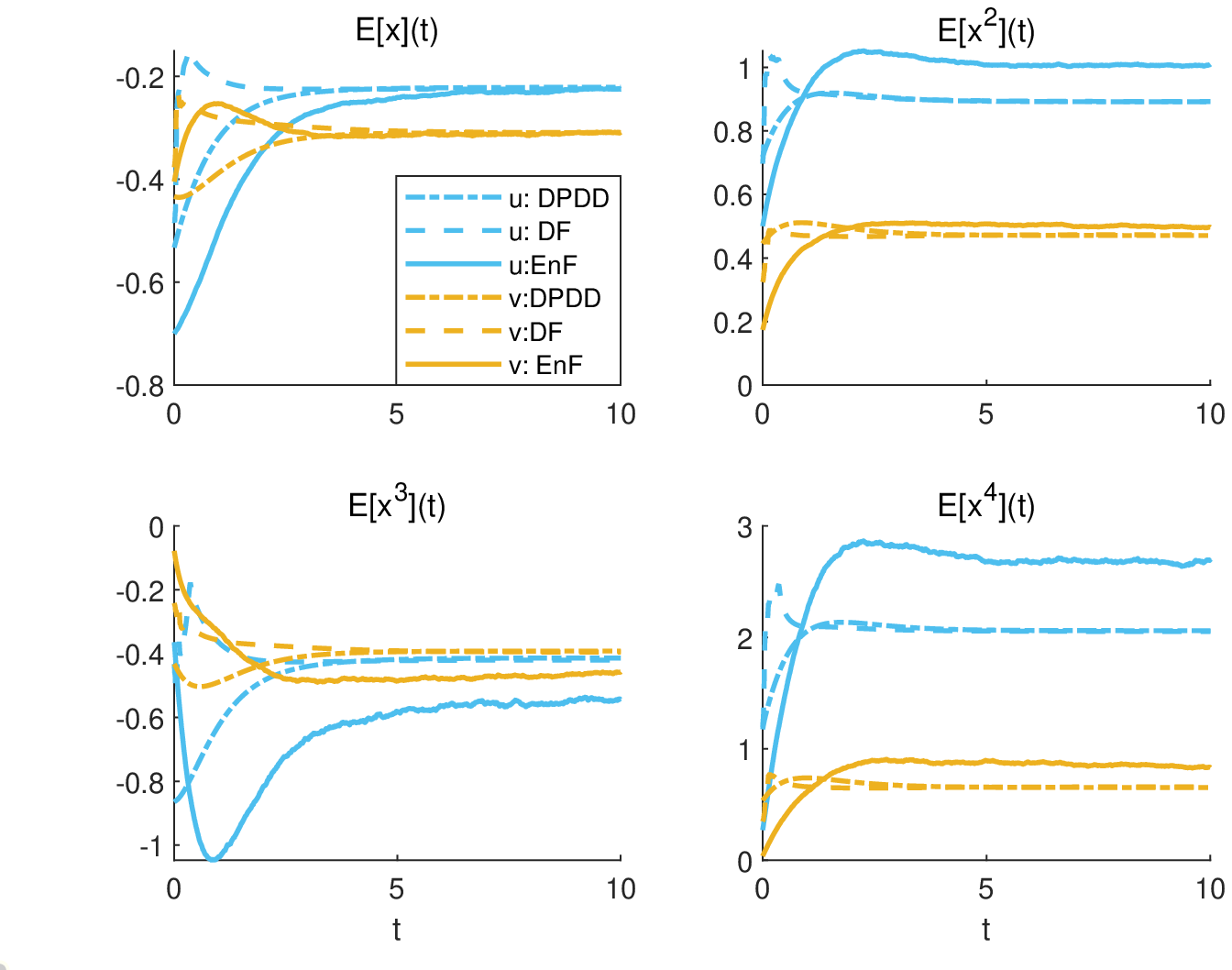}
		\caption{\textit{The evolution of 1st-4th order raw moments of this two-dimensional system. The two dimensions are  respectively depicted in ``blue" and ``yellow". DPDD method corresponds to the dash-dotted line, DF method  is corresponding to the dashed line, and the ensamble forecast, as reference, corresponds to the solid line.}}\label{moments_2d}
	\end{figure}

	From Figure \ref{moments_2d}, we find that as time goes, the both methods converge to almost the  same value  close to the reference. The obvious error exists bacause of the finite-dimensional truncation of the spectral expansion. For this example, we project the probability density function to a 4-dimensional space in DPDD approach, while DF method projects the density to a $1000$-dimensional space. The average time to compute $1000$ DF basis functions  is about $1.1\times 10^3$ seconds, and the one to compute $4$ DPDD basis functions  is   $3.8\times 10^{-3}$ seconds. An clear  comparison is listed in Table \ref{time_2d}, which shows that the DPDD  method is much more computationally efficient  than the DF method when they achieve a similar precision of approximation.
	\begin{table}[h]
		\centering
		\caption{The CPU time comparison of computing the corresponding eigenpairs}\label{time_2d}
		\begin{tabular}{cccc}
			\hline
			Method & DPDD & DF \\
			\hline
			The number of bases & 4 & 1000 \\
			CPU time (s) & $3.8\times 10^{-3}$ & $1.1\times 10^{3}$ \\
			\hline
		\end{tabular}
	\end{table}

	\subsection{Noisy Lorenz-63 model}\label{Lorenz63} In this subsection, we consider a more complex system - noisy Lorenz-63 model, which is intrinsically chaotic and stochastic. Noisy Lorenz-63 system characterizes the atmospheric convection model, and the governing equation  is given by,
	\begin{equation*}
		\left\{
		\begin{aligned}
			\frac{dx}{dt}&=\sigma(y-x)+q_x\eta_x\\
			\frac{dy}{dt}&=x(\rho-z)-y+q_y\eta_y\\
			\frac{dz}{dt}&=xy-\beta z+q_z\eta_z,\\
		\end{aligned}
		\right.
	\end{equation*}
	where the unknown functions $x$, $y$ and $z$ represent the velocity, the horizontal temperature variation, and the vertical temperature variation, respectively. In the paper, we consider the system with standard Lorenz parameters  Prandtl number $\sigma=10$, the (relative) Rayleigh number $\beta=8/3$, and the geometric factor $\rho=28$. Besides, the stochastic force is assumed to be the independent Gaussian white noise with same intensity, i.e., $\eta_xdt,\eta_ydt,\eta_zdt \sim \mathcal{N}(0,dt)$ and $q_x=q_y=q_z=0.1$.  We simulate the trajectories in time interval $[0,100]$ by Runge-Kutta scheme with temporal step-size $\Delta t=0.01$. We model the density by  the DPDD method (with 2 basis functions) and the diffusion forecast (with 1000 basis functions), and the ensemble forecast solution  is used for reference. A 3-dimensional scatter plot, where color of the circle corresponds to the value of density, is drawn in Figure \ref{pdf_lorenz63}. By Figure \ref{pdf_lorenz63}, we notice that the larger value of density is around the attractor.  
	For a better visualization,  we project the 3D scatter plot onto $xOz$ plane. By Figure \ref{pdf_lorenz63} and Figure \ref{pdf_projectlorenz63}, we find that the solution profile of DPDD is different from the solution profile of DF. This may be because of the strong  instability of the chaotic system. The solution pattern  by DPDD and DF looks similar to the ensemble forecast solution. For the simulation, DPDD uses 2 basis functions and DF uses 1000 basis functions. Thus, DPDD is more efficient than DF for any real-time computation.

	\begin{figure}[htbp]
		\centering
		\includegraphics[width=5.5in]{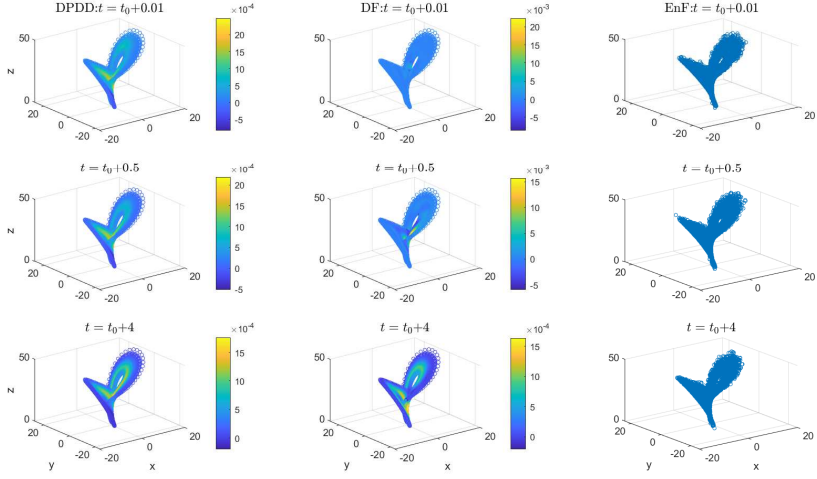}
		\caption{\textit{The approximate probability density function obtained by different methods at different moments. Left(DPDD), middle(DF), right(EnF). }}\label{pdf_lorenz63}
	\end{figure}

	\begin{figure}[htbp]
		\centering
		\includegraphics[width=5.5in]{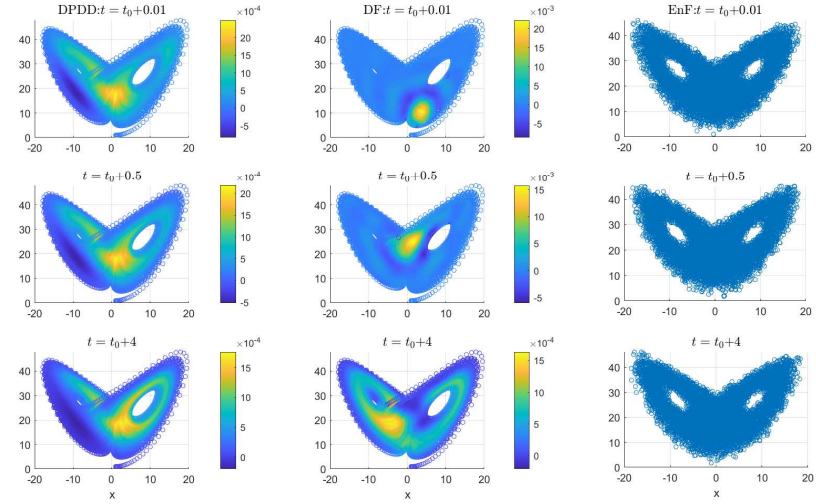}
		\caption{\textit{The projected probability density function obtained by different methods at different moments. Left(DPDD), middle(DF), right(EnF). }}\label{pdf_projectlorenz63}
	\end{figure}
	
	\subsection{An ocean model}\label{realistic} In this subsection, we consider a more practical problem,  the sea surface temperature (SST),  which draws a vast interest in climate field. We have  a global monthly SST dataset, the Extended Reconstructed Sea Surface Temperature (ERSST V5) \cite{sst_data} dataset, which includes the average monthly temperature from January 1854 to June 2022. The data set spatially covers the global (latitude: 88.0N - 88.0S, longitude: 0.0E - 358.0E), and is divided into an 89x180 grid. We take the 600 temporal data from January 1971 to December 2020 as observed data to estimate the eigenpairs, and the data from January 2020 to June 2022 as test data.  Moreover,  we will forecast the temperature for the second half year of 2022. The probability density functions for all the 16020 grid points are computed, which means that the total amount of data is  $9.612\times 10^6$.
	\begin{figure}[htbp]
		\begin{minipage}[t]{0.5\linewidth}
			\centering
			\includegraphics[width=2.8in, height=2.3in]{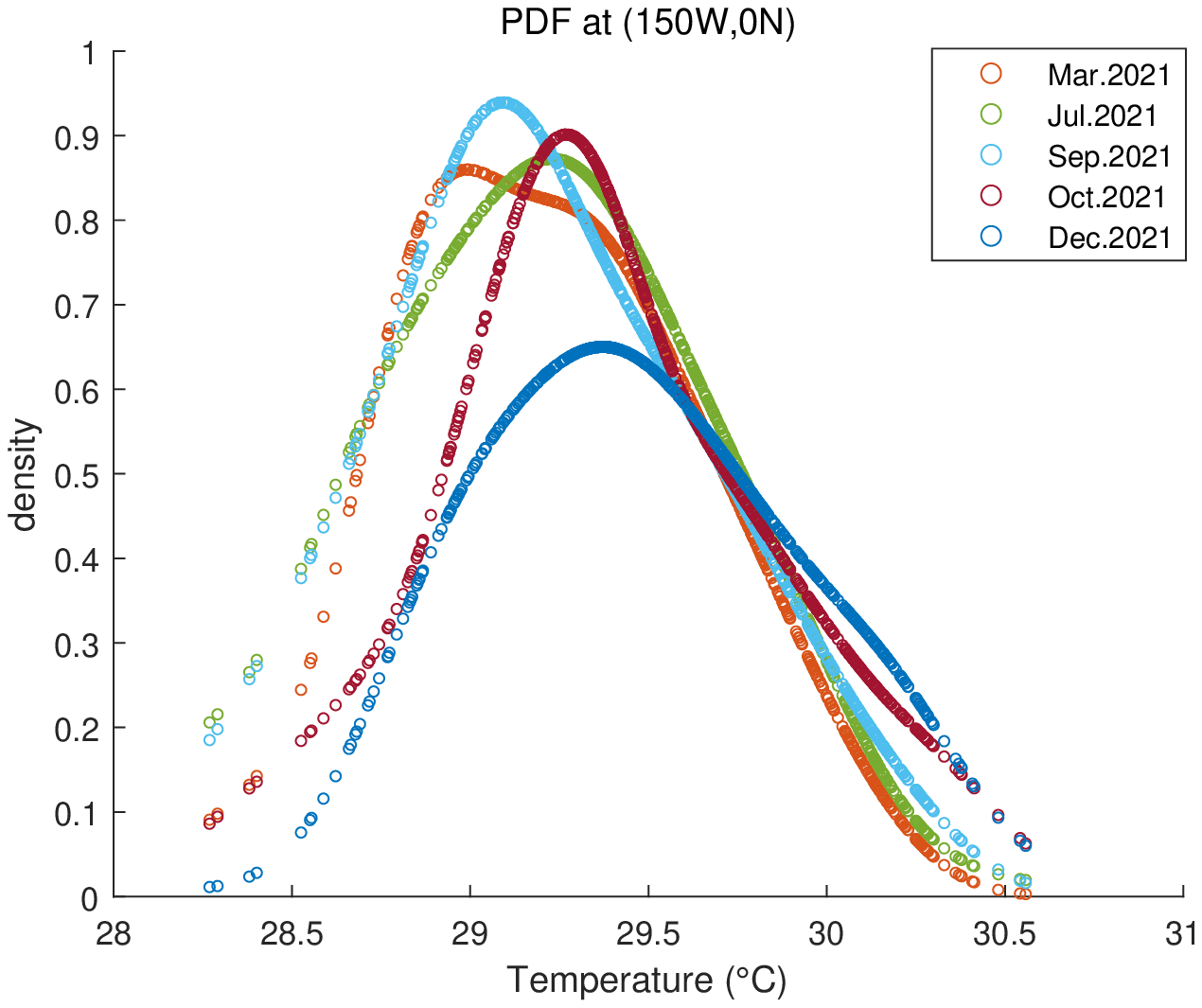}
			\caption{PDF of SST for five different months at 150 west longitude, 0 nouth latitude}\label{pdf_sst1}
		\end{minipage}
		\begin{minipage}[t]{0.5\linewidth}
			\centering
			\includegraphics[width=2.8in, height=2.3in]{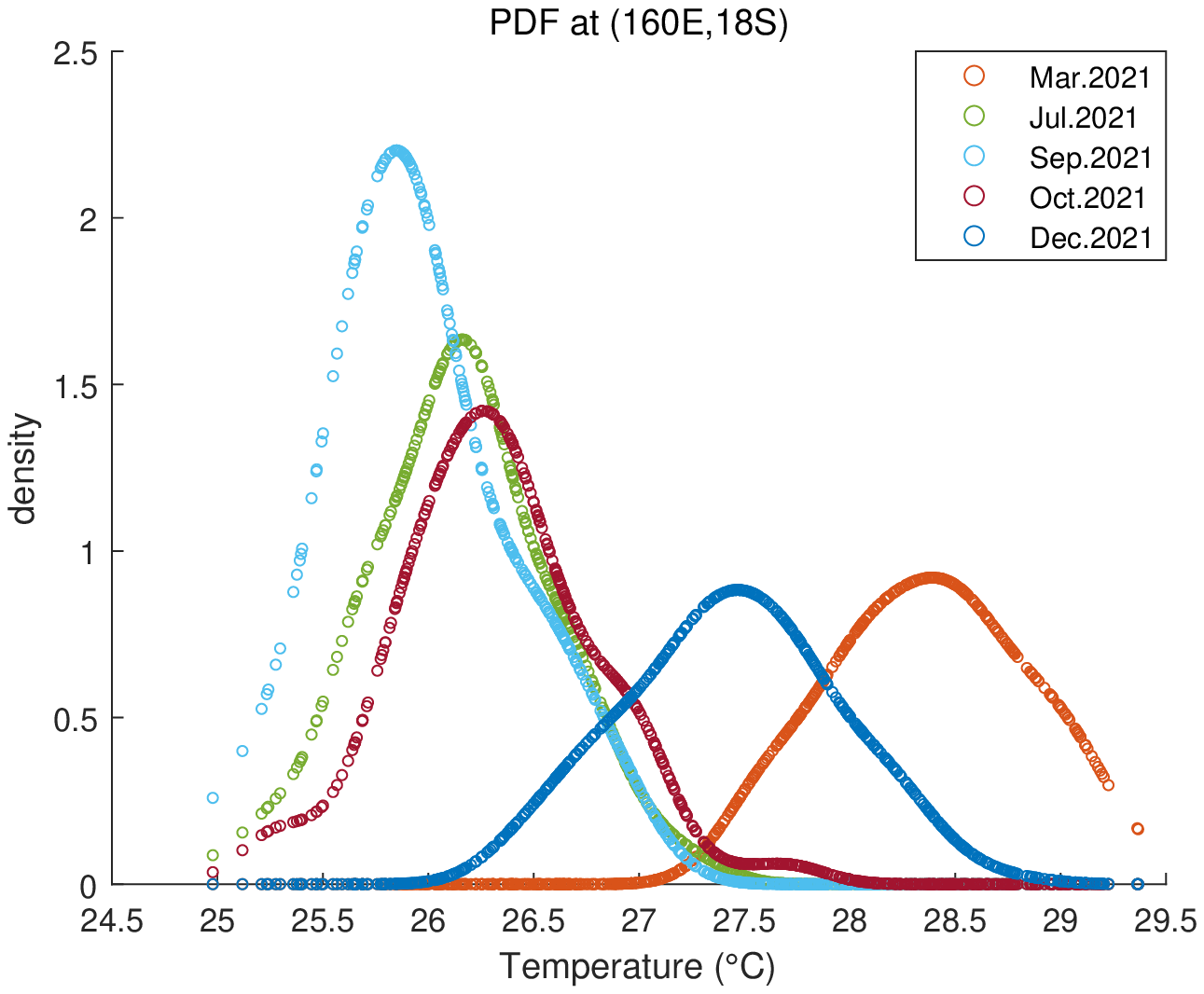}
			\caption{PDF of SST for five different months at 160 east longitude, 18 south latitude}\label{pdf_sst2}
		\end{minipage}
	\end{figure}
	
	The temperature figures at different months are presented to demonstrate the application  of DPDD method solving  the realistic problem. Figure \ref{pdf_sst1} and Figure \ref{pdf_sst2} show the approximate probability density functions of SST for different months in 2021 at two fixed locations, respectively. Comparing Figure \ref{pdf_sst1} and Figure \ref{pdf_sst2}, we find that the density function of average monthly temperature on the equator barely changes throughout the year, while the one in the southern hemisphere varies greatly. On the equator, higher temperatures are maintained all year with mean 29.2 centigrade degrees, and the maximum SST can be over 30 centigrade degrees. Whereas, in the southern hemisphere, temperatures are the highest for March (in summer) and lowest for September (in winter), which is the exact opposite of what happens in the northern hemisphere. The lowest SST can be 25 centigrade degrees in the southern hemisphere and the maximum rarely reaches 29.5 centigrade degrees, with a large temperature difference throughout the year. This indicates that the sea surface temperatures in southern and northern hemispheres are quite different from the equator. Then we approximate global SST for February 2020 and June 2021 comparing to true data, and also make a prediction for December 2022.
	\begin{figure}[htbp]
		\begin{minipage}[t]{0.5\linewidth}
			\centering
			\includegraphics[width=2.8in, height=2.3in]{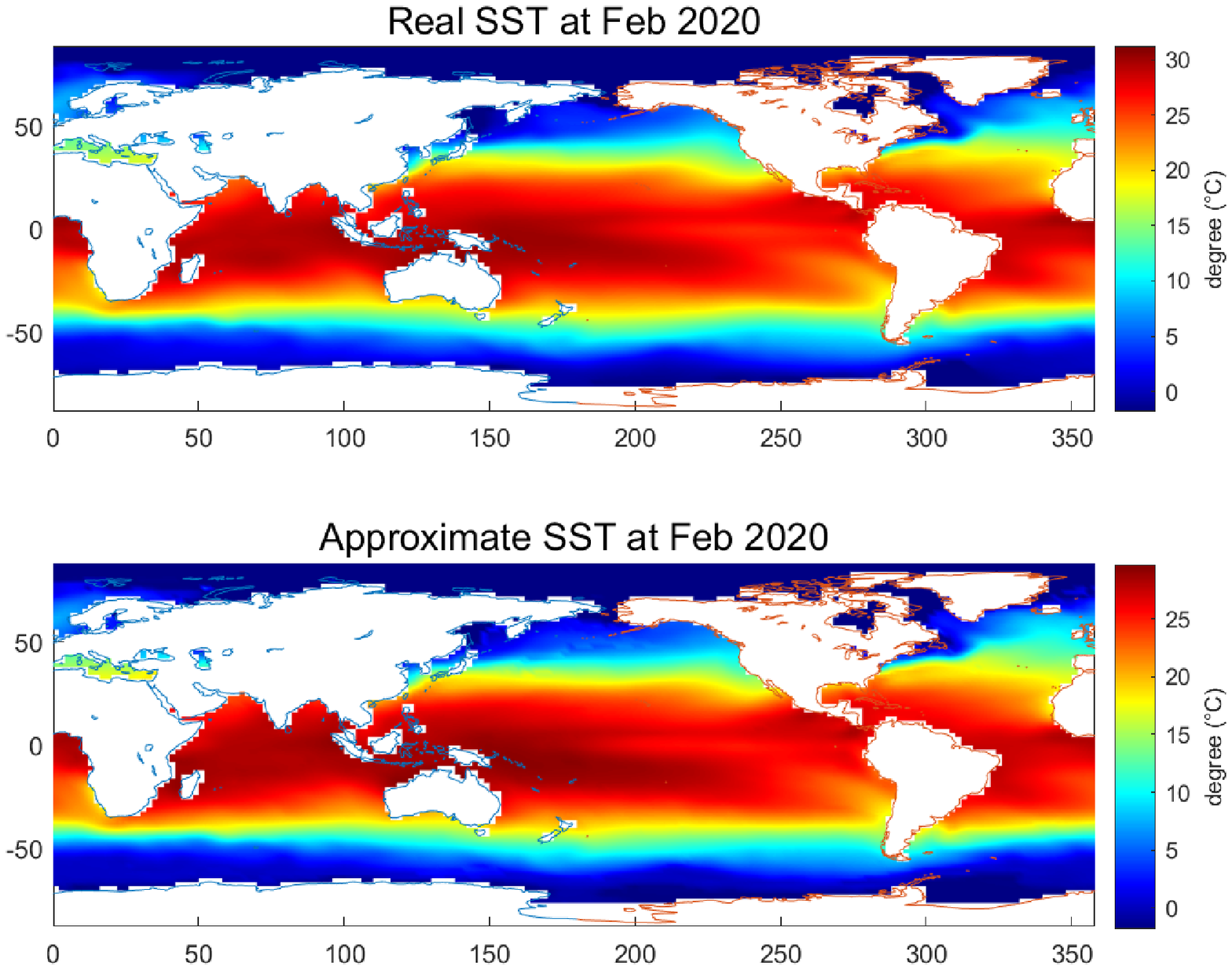}
			\caption{Global SST for February 2020, real data (top), approximate value (bottom).}\label{SST1}
		\end{minipage}
		\begin{minipage}[t]{0.5\linewidth}
			\centering
			\includegraphics[width=2.8in, height=2.3in]{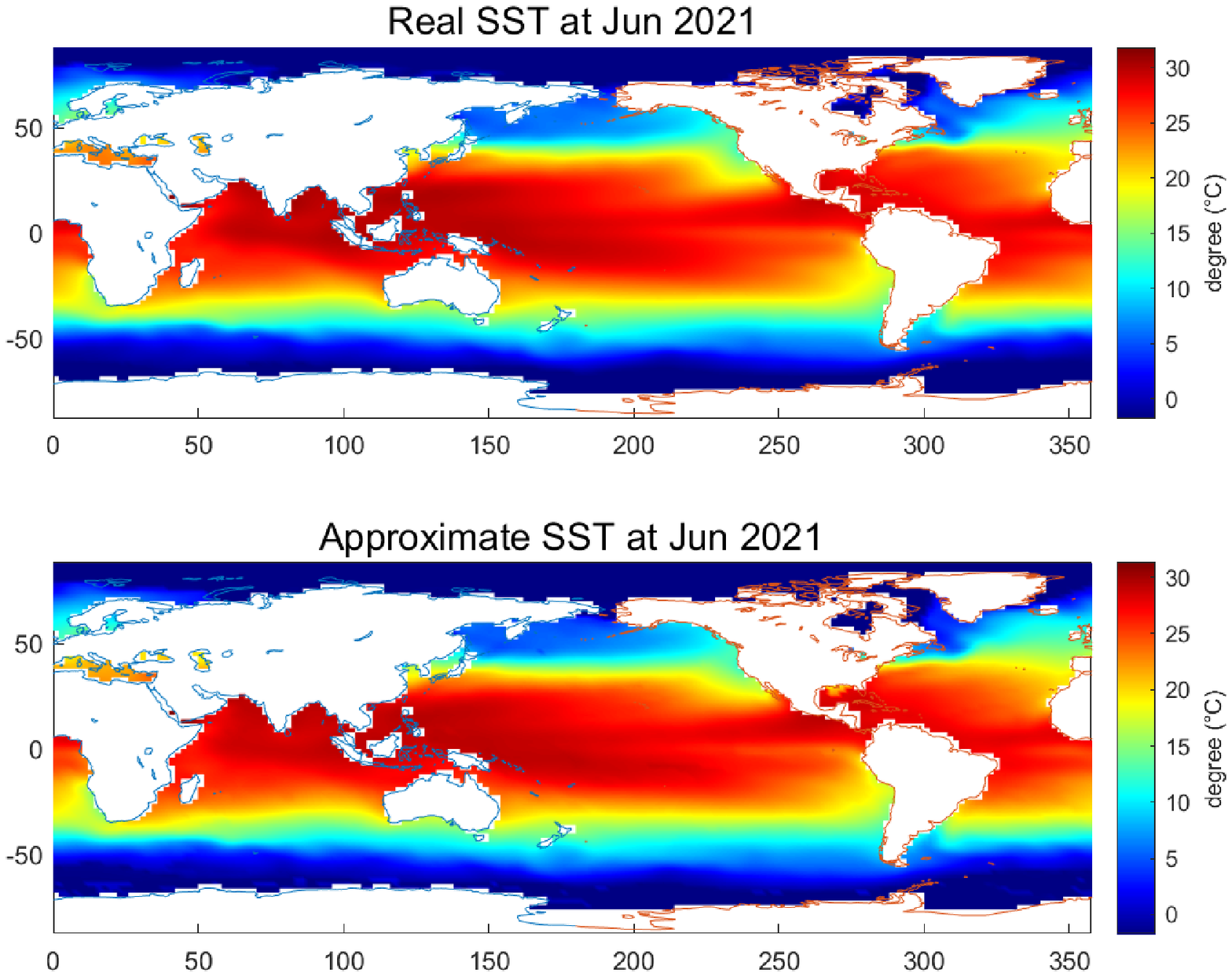}
			\caption{Global SST for September 2021, real data (top), approximate value (bottom).}\label{SST2}
		\end{minipage}
	\end{figure}
	\begin{figure}[htbp]
		\centering
		\includegraphics[width=5.5in, height=2.5in]{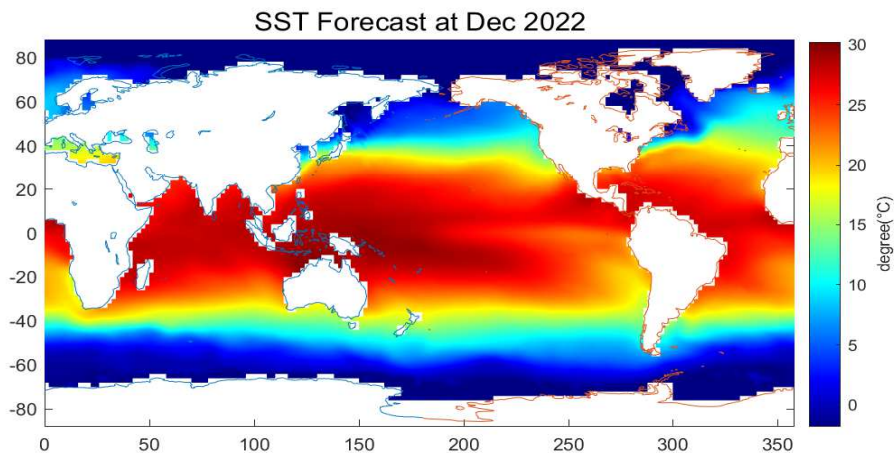}
		\caption{Global SST forecast for December 2022}\label{SST_pre}
	\end{figure}
	\par	As depicted in Figure \ref{SST1} and Figure \ref{SST2}, the approximate global SST is very close to the reference value which is the real data from the dataset (ERSST v5). It is worth mentioning that the climate of the Mediterranean Sea in February is warm while it is torrid in June. In contrast, the sea surface temperature of the South Atlantic Ocean is higher in February than in June.  And the prediction for December 2022 is illustrated in Figure \ref{SST_pre}. These figures show  that the annual variation of SST in near-land areas is significantly larger than that in areas far from land, and the temperatures in southern and northern hemispheres change oppositely. DPDD sorts out this practical problem of having prodigious amounts of data, estimates the probability density functions well and predicts the state well, thus  helps to make a scientific   decision.

	\section{Conclusion}\label{con}
	In this work, we have presented a nonparametric method, dynamic probability density decomposition (DPDD),  to approximate the probability density functions for stochastic dynamical systems utilizing the data-driven mechanism.  The snapshots of time-series data were sampled from the  stationary distribution  and can be  one long-time trajectory or multiple trajectories.  This approach was based on a stochastic Koopman formalism, operates on the snapshots, and estimated the probability density functions represented by  the linear expansion of orthonormal basis functions in a weighted $L^2$ space. This method has enabled us to represent the density functions as a finite-term truncation, and predicted the evolution of probability density for arbitrary time, and also the expectation values of observables. As a byproduct of DPDD, one can decompose arbitrary functions of dynamical state using the eigenfunctions of the infinitesimal generator of SDE, and the decomposition is similar to DPDD. Thus we can construct a dimensional reduction on data manifold. Comparing with diffusion forecast,  DPDD behaviors  better  in applicability,  model accuracy and computation efficiency. Moreover, we have given a rigourous  convergence analysis of this method. Finally, DPDD was used to make probability density forecast for  one-dimensional OU process, quadratic turbulence system, noisy-Lorenz-63, and a practical  problem of global sea surface temperature.

\end{document}